\documentclass[12pt]{amsart}
\usepackage[utf8]{inputenc}
\usepackage{amsthm}
\usepackage[T2A]{fontenc}
\usepackage[english,]{babel}
\usepackage{graphicx}
\usepackage{amsfonts}
\usepackage{amsmath}
\usepackage{amssymb}
\usepackage{amscd}
\usepackage[all,cmtip,matrix, arrow]{xy}
\usepackage{amsthm}
\usepackage[left=1.2cm,right=1.5cm,top=2cm,bottom=3cm]{geometry}
\usepackage{tikz-cd}

\usepackage{hyperref}
\usepackage{color}

\newcommand{\bS}{\mathbf{S}}

\newcommand{\LL}{\mathbb{L}}
\newcommand{\PP}{\mathbb{P}}
\newcommand{\RR}{\mathbb{R}}
\newcommand{\ZZ}{\mathbb{Z}}

\newcommand{\fE}{\mathfrak{E}}
\newcommand{\fG}{\mathfrak{G}}

\newcommand{\cO}{\mathcal{O}}
\newcommand{\cS}{\mathcal{S}}
\newcommand{\cT}{\mathcal{T}}
\newcommand{\cU}{\mathcal{U}}
\newcommand{\cD}{\mathcal{D}}

\newcommand{\cR}{\mathcal{R}}

\newcommand{\Db}{\mathrm{D}^{b}}

\DeclareMathOperator{\IGr}{IGr}
\DeclareMathOperator{\Gr}{Gr}
\DeclareMathOperator{\GL}{GL}
\DeclareMathOperator{\Sp}{Sp}

\newtheorem{lemma}{Lemma}[section]
\newtheorem{theorem}[lemma]{Theorem}
\newtheorem{proposition}[lemma]{Proposition}

\newtheorem{corollary}[lemma]{Corollary}

\theoremstyle{definition}
\newtheorem{definition}[lemma]{Definition}

\newtheorem{remark}[lemma]{Remark}
\newtheorem{example}[lemma]{Example}
\newtheorem{conjecture}[lemma]{Conjecture}
\title{On the Derived Category Of $\IGr(3,8)$}
\author{Lyalya Guseva}
\thanks{The author was partially supported by the Russian Academic Excellence Project `5-100' and the Moebius Contest Foundation for Young Scientists}
\address{Laboratory of Algebraic Geometry, National Research University Higher School of Economics,
Russian Federation}
\email{lyalya.guseva1994@gmail.com}

\begin{document}

\maketitle{}

\begin{abstract}
We construct a full exceptional collection of vector bundles in the bounded derived category of coherent sheaves
on the Grassmannian $\IGr(3,8)$ of isotropic 3-dimensional subspaces in a symplectic vector space of dimension~8.
\end{abstract}

\section{Introduction}
The bounded derived category of coherent sheaves is one of the most important invariants of an algebraic variety.
This is one of the reasons to investigate its structure.
In general, the structure of a triangulated category may be quite complicated.
However, there is an important case when it can be described fairly explicitly, namely the case when a triangulated category
possesses a full exceptional collection~$(E_{1}, E_{2}, \ldots, E_{m})$.
In this case every object of a triangulated category admits a unique
filtration with $i$-th subquotient being a direct sum of shifts of the objects $E_i$.
Therefore, an exceptional collection can be considered as a kind of basis for triangulated category.

The simplest example of a variety with a full exceptional collection is a projective space.
Beilinson~\cite{6}~in~1978 showed that the collection of line bundles $\mathcal{O},\mathcal{O}(1), \ldots, \mathcal{O}(n)$  on $\mathbb{P}^n$
is a full exceptional collection.
In 1988~ Kapranov~\cite{7} constructed full exceptional collections on Grassmannians and flag varieties of groups $\operatorname{SL}_{n}$
and on smooth quadrics.
It has been conjectured afterwards that:

\begin{conjecture} \label{conjecture}
If $G$ is a semisimple algebraic group over an algebraically closed field of characteristic~$0$ and~$P\subset G$ is a parabolic subgroup
then there is a full exceptional collection of vector bundles on~$G/P$.
\end{conjecture}

The conjecture easily reduces to the case when $G$ is a simple group and $P$ is its maximal parabolic subgroup, see~\cite[Section~1.2]{1}.
In the case of simple $G$ and maximal $P$
the conjecture is known to be true for the following series
(we use the Bourbaki indexing of simple roots):
\begin{itemize}
\item $G$ is of Dynkin type $A$ and any $P$, see \cite{7};
\item $G$ is of Dynkin type $B$ and $P$ corresponds to one of the first two simple roots, see~\cite{7,2};
\item $G$ is of Dynkin type $C$ and $P$ corresponds to one of the first two simple roots, see~\cite{6,2};
\item $G$ is of Dynkin type $D$ and $P$ corresponds to the first simple root, see~\cite{7};
\end{itemize}
and for the following sporadic cases:
\begin{itemize}
\item $(B_3,P_3)$, $(B_4,P_4)$, see~\cite{7,K06};
\item $(C_3,P_3)$, $(C_4,P_4)$, $(C_5,P_5)$, see~\cite{11,S01};
\item $(D_4,P_3)$, $(D_4,P_4)$, $(D_5,P_4)$, $(D_5,P_5)$, see~\cite{7,K06};
\item $(E_6,P_1)$, $(E_6,P_6)$, see~\cite{FM};
\item $(G_2,P_1)$, $(G_2,P_2)$, see~\cite{7,K06}.
\end{itemize}
Besides that, an exceptional collection of maximal possible length (equal to the rank of the Grothendieck group)
has been constructed on $G/P$ for all classical groups (i.e. for groups of Dynkin types $ABCD$)
and all their maximal parabolic subgroups, see~\cite{1};
however the fullness of these collections is not yet known.

In this work we discuss the first unknown case for the symplectic group.
This is the case $(C_4,P_3)$, where the group is the symplectic group $G=\mathrm{Sp}(8)$
and the parabolic subgroup corresponds to the third simple root.
The corresponding homogeneous space $\IGr(3,8)$ is the Grassmannian of 3-dimensional isotropic subspaces
in an 8-dimensional symplectic vector space.
The other cases of maximal parabolic subgroups of $\mathrm{Sp}(8)$,
i.e., $\IGr(1,8) \cong \PP^7$, $\mathrm{IGr}(2,8)$, and $\mathrm{IGr}(4,8)$, were established in \cite{6,2,11}.

The exceptional collection on $\IGr(3,8)$ that we construct is a \emph{Lefschetz collection}, \cite{2,K06,8}.
Recall that a Lefschetz exceptional collection with respect to a line bundle $\mathcal{L}$ is just
an exceptional collection which consists of several blocks,
each of them is a sub-block of the previous one twisted by $\mathcal{L}$, see Definition~\ref{def:lefschetz} for more details.
If all blocks are the same, the collection is called \emph{rectangular}.

Let $\cU$ denote the tautological sub-bundle on $\mathrm{IGr}(3,8)$.
Denote by $\fE$ and $\fE'$ the following collections of vector bundles on $\IGr(3,8)$:
\begin{align}
\label{eq:e}
\mathfrak{E}\hphantom{{}'}	&:=	
(\hphantom{\Sigma^{2,1}\cU^{\vee}(-1),{}}\mathcal{O},\cU^{\vee},S^2\cU^{\vee},\Lambda^{2}\cU^{\vee},\Sigma^{2,1}\cU^{\vee}),\\
\label{eq:e-prime}
\mathfrak{E}'	&:=	
(\Sigma^{2,1}\cU^{\vee}(-1),\mathcal{O},\cU^{\vee},S^2\cU^{\vee},\Lambda^{2}\cU^{\vee}\hphantom{{},\Sigma^{2,1}\cU^{\vee}}),
\end{align}
where $\Sigma^{2,1}\cU^\vee \cong (\cU^\vee \otimes S^2\cU^\vee)/S^3\cU^\vee$.
We will denote by $\fE(i)$ and $\mathfrak{E}'(i)$ the collections consisting of the corresponding
five vector bundles twisted by $\cO(i)$,
and in the same way the subcategories of $\Db(\IGr(3,8))$ generated by these.
We will also need a vector bundle $\Sigma^{3,1}\cU^\vee \cong (\cU^\vee \otimes S^3\cU^\vee)/S^4\cU^\vee$
and we denote by $\LL$ and $\RR$ the left and right mutation functors, see the precise definition in~Section~\ref{section:preliminaries}.


The main result of this article is the following theorem.

\begin{theorem}
\label{intro:main}
The objects
\begin{equation*}
T := (\mathbb{L}_{\mathfrak{E}}(\Sigma^{3,1}\cU^{\vee}))[-3]
\qquad\text{and}\qquad
F := \RR_{\Sigma^{2,1}\cU^\vee(-1)}(T)
\end{equation*}
are equivariant vector bundles on $\IGr(3,8)$.

The collections of $32$ vector bundles  on $\IGr(3,8)$
\begin{equation*}
\begin{array}{lllllllll}
F,& \mathfrak{E}\hphantom{{}'},& F(1), & \mathfrak{E}\hphantom{{}'}(1),& \mathfrak{E}\hphantom{{}'}(2), &
\mathfrak{E}\hphantom{{}'}(3),& \mathfrak{E}\hphantom{{}'}(4),& \mathfrak{E}\hphantom{{}'}(5), & \qquad\text{and}\\
T,& \mathfrak{E}',& T(1), & \mathfrak{E}'(1),& \mathfrak{E}'(2),& \mathfrak{E}'(3),& \mathfrak{E}'(4),& \mathfrak{E}'(5)
\end{array}
\end{equation*}
are full Lefschetz collections with respect to the line bundle $\cO(1)$.

The collections of $32$ vector bundles  on $\IGr(3,8)$
\begin{equation*}
\begin{array}{lllllllll}
F,& \mathfrak{E}\hphantom{{}'},& \mathfrak{E}\hphantom{{}'}(1),& \mathfrak{E}\hphantom{{}'}(2),&
F(3),& \mathfrak{E}\hphantom{{}'}(3),& \mathfrak{E}\hphantom{{}'}(4),& \mathfrak{E}\hphantom{{}'}(5), & \qquad\text{and} \\
T,& \mathfrak{E}',& \mathfrak{E}'(1),& \mathfrak{E}'(2),& T(3),& \mathfrak{E}'(3),& \mathfrak{E}'(4),& \mathfrak{E}'(5)
\end{array}
\end{equation*}
are full rectangular Lefschetz collections with respect to the line bundle~$\cO(3)$.
\end{theorem}


Descriptions of the vector bundles $T$ and $F$ can be found in Lemma~\ref{T_1} and Remark~\ref{remark:f} respectively.
In particular, the bundle $F$ is isomorphic to a twist of the vector bundle $\mathcal{E}^{2,0,0;1}$ constructed in~\cite{1}.

A significant part of the proof of Theorem~\ref{intro:main} is based on the study of a certain interesting bicomplex
{\tiny\begin{equation*}
\label{eq:first-left}
\vcenter{\xymatrix@C=.9em{
&
0 \ar[r] &
\Sigma^{3,2}\cU^{\vee}(-3) \ar[r] &
V\otimes \Sigma^{2,1}\cU^{\vee}(-2) \ar[r] &
\Lambda^2 V\otimes\cU^{\vee}(-1) \ar[r] &
\Lambda^4V\otimes\mathcal{O} \ar[r] &
\Lambda^2V\otimes \Lambda^2\cU^{\vee} \ar[r] &
V\otimes \Sigma^{2,1}\cU^{\vee} \ar[r] &
\Sigma^{3,1}\cU^{\vee} \ar[r] &
0
\\
0 \ar[r] &
\Sigma^{3,3}\cU^{\vee}(-4) \ar[r] \ar[u] &
V\otimes\Sigma^{2,2}\cU^{\vee}(-3) \ar[u] \ar[r] &
\Lambda^2 V\otimes \Lambda^2\cU^{\vee}(-2) \ar[u] \ar[r] &
\Lambda^3 V\otimes \cO(-1) \ar[u] \ar[r] &
\Lambda^2V\otimes\mathcal{O} \ar[u] \ar[r] &
V\otimes \cU^{\vee}\ar[u] \ar[r] &
S^2\cU^{\vee} \ar[u] \ar[r] &
0, \ar[u]
}}
\end{equation*}}%
of vector bundles on $\IGr(3,8)$, where $V$ is the tautological 8-dimensional representation of $\operatorname{Sp}(8)$.
This bicomplex is $\operatorname{Sp}(8)$-equivariant, its lines are exact and are obtained as the restrictions
of the so-called staircase complexes (see~\cite{4}) from $\Gr(3,8)$.
The vector bundle $T$ is identified with the cohomology of the truncation~\eqref{eq:first-right} or~\eqref{eq:first-left} of this bicomplex,
and using the bicomplex we prove an isomorphism
\begin{equation*}
\mathbb{L}_{\mathfrak{E}'(1),\mathfrak{E}'(2)}(T(3))=T(1)[4],
\end{equation*}
which is crucial for the proof of completeness of the above exceptional collections.
We want to stress that this part of the argument is similar to the one used in~\cite{2} in the case of $\IGr(2,2n)$;
so it seems likely that an analogous construction can be used for other homogeneous varieties.

To prove the fullness of the exceptional collections in Theorem~\ref{intro:main} we first prove that some special objects
lie in the subcategory $\mathcal{D}$ of $\mathrm{D}^b(\mathrm{IGr}(3,8))$ generated by each of these collections.
After that we consider the isotropic flag variety $\mathrm{IFl}(2,3;8)$ with its two projections
\begin{equation*}
\xymatrix{
& \mathrm{IFl}(2,3;8) \ar[dl] \ar[dr]
\\
\IGr(2,8) &&
\IGr(3,8)
}
\end{equation*}
The first arrow is a $\PP^3$-fibration.
Using a certain variant of the Lefschetz exceptional collection on $\IGr(2,8)$ from~\cite{2} and Orlov's projective bundle formula
we construct a very special full exceptional collection on~$\mathrm{IFl}(2,3;8)$.
The main property of this exceptional collection is that the pushforwards along the second arrow (which is a $\PP^2$-fibration)
of almost all objects constituting it are contained in the subcategory~$\cD$, and for the few objects that do not enjoy this property,
the pushforwards are contained in the subcategory~$\fE(6) \subset \Db(\IGr(3,8))$.
It follows from this that every object of $\Db(\IGr(3,8))$ contained in the orthogonal~${}^\perp\cD$ to the subcategory $\cD$,
belongs to $\fE(6)$.
The trivial observation
\begin{equation*}
{}^\perp \fE \cap \fE(6) = 0
\end{equation*}
(that follows immediately from the Serre duality on $\IGr(3,8)$) then shows that ${}^\perp\cD = 0$,
and completes the proof of the fullness of the collections.


The work is organized as follows.
In Sections~\ref{section:preliminaries} and~\ref{section:bbw} we collect some preliminary results
from the theory of derived categories and equivariant vector bundles on Grassmannians.
In  Section~\ref{section:db} we prove vanishing lemmas that are essential for the proof of exceptionality and fullness of the constructed collection.
In Section~\ref{section:exact} we collect some important exact sequences and construct the bicomplex discussed above.
Also in this section we prove some important properties of this bicomplex.
In Section~\ref{section:non-rectangular} we construct vector bundles $F$ and $T$
and prove exceptionality of the collection of Theorem~\ref{intro:main}.
In Section~\ref{section:fullness} we give a proof of fullness of the constructed collection.
Finally, in Section~\ref{section:applications} we provide a couple of applications of our results:
compute the residual category of $\IGr(3,8)$ as defined in~\cite{KS},
and construct a pair of (fractional) Calabi--Yau categories related to a half-anticanonical section
and anticanonical double covering of $\IGr(3,8)$.

{\bf Acknowledgements:}
I would like to thank my advisor, Alexander Kuznetsov, for suggesting this problem as well as for his patience and constant support, and Anton Fonarev for useful comments on the draft of this paper.
%

\section{Preliminaries}
\label{section:preliminaries}

Let $\Bbbk$ be an algebraically closed field of characteristic zero and let $\mathcal{T}$ be a $\Bbbk$-linear triangulated category.
We start by recalling some basic definitions.

\begin{definition}
A sequence of full triangulated subcategories $\mathcal{A}_1, \ldots ,\mathcal{A}_m \in \mathcal{T}$ is \textbf{semiorthogonal}
if for all $0\le i {}<{} j\le m$ and all $G\in \mathcal{A}_i$, $H\in \mathcal{A}_j$ one has $\mathrm{Hom}_{\mathcal{T}} (H,G) = 0$.
Let $\langle \mathcal{A}_1, \ldots ,\mathcal{A}_m\rangle$ denote
the smallest full triangulated subcategory in $\mathcal{T}$ containing all $\mathcal{A}_i$.
If $\langle \mathcal{A}_1, \ldots ,\mathcal{A}_m \rangle = \mathcal{T}$,
we say that the subcategories $\mathcal{A}_i$ form a \textbf{semiorthogonal decomposition} of $\mathcal{T}$.
\end{definition}

\begin{definition}
An object $E$ of  $\mathcal{T}$ is \textbf{exceptional} if $\mathrm{Ext}^{\bullet}(E, E) = \Bbbk$
(that is, $E$ is simple and has no non-trivial self-extensions).
\end{definition}

If $E$ is exceptional, the minimal triangulated subcategory $\langle E \rangle$ of $\cT$ containing $E$
is equivalent to~$\Db(\Bbbk)$, the bounded derived category of $\Bbbk$-vector spaces,
via the functor $\Db(\Bbbk) \to \cT$ that takes a graded vector space $V$ to $V \otimes E \in \cT$.

\begin{definition}
A sequence of objects $E_{1},\ldots, E_{m}$ in  $\mathcal{T}$ is an \textbf{exceptional collection}
if each $E_{i}$ is exceptional and $\mathrm{Ext}^{\bullet}(E_{i}, E_{j} ) = 0$ for all $i > j$.
A collection $(E_{1},E_{2},\ldots ,E_{m})$ is \textbf{full} if the minimal triangulated subcategory of $\mathcal{T}$
containing $(E_{1},E_{2},\ldots, E_{m})$ coincides with $\mathcal{T}$.
\end{definition}

For an exceptional object $E \in \cT$ we denote by $\LL_E$ and $\RR_E$ the \textbf{left} and \textbf{right mutation} functors through $E$,
which are defined as taking an object $G \in \cT$ to
\begin{equation*}
\mathbb{L}_{E}(G):=\mathrm{Cone}(\mathrm{Hom}^{\bullet}(E, G)\otimes E \to G),
\qquad\text{and}\qquad
\mathbb{R}_{E}(G):= \mathrm{Cone}(G \to \mathrm{Hom}^{\bullet}(G,E)^{\vee} \otimes E)[-1],
\end{equation*}
where the morphisms are given by the evaluation and coevaluation, respectively.


It is well known (see \cite{3}) that if $(E,E')$ is an exceptional pair then $(E',\mathbb{R}_{E'}(E))$ and $(\mathbb{L}_{E}(E'),E)$
are also exceptional pairs each of which generates the same subcategory in $\mathcal{T}$ as the initial pair $(E,E').$

More generally, if $(E_1, \ldots , E_m)$ is an exceptional collection of arbitrary length in $\mathcal{T}$ then
one can define the left and right mutations of an object $E\in \mathcal{T}$ through the category $\langle E_1, \ldots , E_n\rangle$
as the compositions of the corresponding mutations through the generating objects:
\begin{equation*}
\LL_{\langle E_1, \ldots , E_m\rangle} = \LL_{E_1} \circ \ldots \circ \LL_{E_m},
\qquad
\RR_{\langle E_1, \ldots , E_m\rangle} = \RR_{E_m} \circ \ldots \circ \RR_{E_1}.
\end{equation*}


\begin{proposition}[\cite{3}]
\label{proposition:mutations}
The functors of left and right mutations through an exceptional collection
induce mutually inverse equivalences of the left and the right orthogonals to the collection:
\begin{equation*}
\xymatrix@1@C=7em{
{}^\perp\langle E_1, \ldots , E_m\rangle \ar@<.5ex>[r]^{\LL_{\langle E_1, \ldots , E_m\rangle}} &
\langle E_1, \ldots , E_m\rangle^\perp \ar@<.5ex>[l]^{\RR_{\langle E_1, \ldots , E_m\rangle}}
}
\end{equation*}
Mutation of a \textup(full\textup) exceptional collection is a \textup(full\textup) exceptional collection.
\end{proposition}

Let $X$ be a smooth projective algebraic variety over a field $\Bbbk$.
We denote by $\mathrm{D}^{b}(X)$ the bounded derived category of coherent sheaves on $X$.
The canonical line bundle of~$X$ is denoted by $\omega_X$.
The following result is useful when dealing with exceptional collections on $X$.

\begin{proposition}
\label{proposition:long-mutations}
If $(E_1, E_2, \ldots, E_{m-1}, E_m)$ is an exceptional collection in $\mathrm{D}^{b}(X)$ then
\begin{equation*}
(E_2, \ldots, E_{m-1}, E_m, E_1 \otimes \omega_X^{-1})
\qquad\text{and}\qquad
(E_m \otimes \omega_X, E_1, E_2, \ldots, E_{m-1})
\end{equation*}
are exceptional collections too.
If one of these collections is full then so are the others.
\end{proposition}
\begin{proof}
The first part follows easily from Serre duality.
The second part is~\cite[Theorem 4.1]{3}.
\end{proof}

%


The next definition is also quite useful.

\begin{definition}[\cite{8,2}]
\label{def:lefschetz}
Let $\mathcal{L}$ be a line bundle on $X$.
\begin{enumerate}
\item A \textbf{Lefschetz collection} in $\mathrm{D}^{b}(X)$ with respect to a line bundle
$\mathcal{L}$ is an exceptional collection of objects of $\mathrm{D}^{b}(X)$ which has a block structure
\begin{equation*}
\underbrace{E_{1}, E_{2},\ldots ,E_{\lambda_{0}}}_{\mbox{block $1$}},
\underbrace{E_{1}\otimes\mathcal{L}, E_{2}\otimes\mathcal{L},\ldots,E_{\lambda_{1}}\otimes \mathcal{L}}_{\mbox{block $2$}},\ldots,
\underbrace{E_{1}\otimes \mathcal{L}^{\otimes{i-1}}, E_{2}\otimes \mathcal{L}^{\otimes i-1},\ldots,E_{\lambda_{i-1}}\otimes \mathcal{L}^{\otimes i-1}}_{\mbox{block $i$}} ,
\end{equation*}
where $\lambda = (\lambda_{0}\ge \lambda_{1}\ge \ldots \ge \lambda_{i-1} > 0)$ is a non-increasing sequence of positive integers that is called the \textbf{support partition} of the Lefschetz collection.
\item If $\lambda_{0}= \lambda_1 = \dots = \lambda_{i-1}$, then the corresponding Lefschetz collection is called \textbf{rectangular}. Otherwise, its \textbf{rectangular part} is the subcollection
    $$E_{1}, E_{2},\ldots,E_{\lambda_{i-1}},E_{1}\otimes\mathcal{L}, E_{2}\otimes\mathcal{L},\ldots,E_{\lambda_{i-1}}\otimes\mathcal{L},\ldots,E_{1}\otimes \mathcal{L}^{\otimes i-1}, E_{2}\otimes \mathcal{L}^{\otimes i-1},\ldots,E_{\lambda_{i-1}}\otimes \mathcal{L}^{\otimes i-1}.$$
\end{enumerate}
\end{definition}

We should point out that being Lefschetz is not a property of an exceptional collection, but rather a structure expressed as the block decomposition.

\begin{example} \label{14}
Now let us give several examples of Lefschetz exceptional collections:
\begin{enumerate}
\item Any exceptional collection can be considered as a $1$-block Lefschetz collection.
\item For any $d>0$ the standard exceptional collection $(\cO_{\mathbb{P}^{n}},\cO_{\mathbb{P}^{n}}(1), \ldots ,\cO_{\mathbb{P}^{n}}(n))$
on $\mathbb{P}^{n}$ is Lefschetz with respect to $\cO_{\mathbb{P}^{n}}(d)$ with support partition
\begin{equation*}
\lambda = (\underbrace{d,d,\ldots,d}_{q},r),
\end{equation*}
where $q$ and $r$ are defined from the equality $n + 1 = qd + r$ with $0<r\le d$.

\item The Lefschetz exceptional collection on the isotropic Grassmannian $X=\mathrm{IGr}(2,V)$ of two-dimensional subspaces
in a symplectic vector space $V$ was constructed in \cite{2}.
If the dimension of~$V$ is equal to $2m$ then the first block of this collection looks like
    $$(\cO_{X},\cU^{\vee},S^2\cU^{\vee}, \ldots, S^{m-1}\cU^{\vee}),$$
where $\cU$ is the tautological bundle, and its support partition is
    $$ \lambda = (\underbrace{m,m,\ldots,m}_{m-1},\underbrace{m-1,m-1, \ldots, m-1}_{m}).$$

\end{enumerate}
\end{example}

Also we will need the following theorem.
\begin{theorem}[Orlov’s projectivization formula, \cite{9}] \label{Orlov}
Let $\mathcal{E}$ be a vector bundle on $X$ of rank n.
Let~$\pi \colon \mathbb{P}(\mathcal{E})\to X$ be the projectivization of $\mathcal{E}$, and let $\cO(1)$ denote
the Grothendieck invertible sheaf.
Then for each $i \in \mathbb{Z}$ there is a semiorthogonal decomposition
\begin{equation*}
\mathrm{D}^{b}(\mathbb{P}(\mathcal{E})) =
\langle \pi^*\mathrm{D}^{b}(X) \otimes \cO(i), \pi^*\mathrm{D}^{b}(X) \otimes \cO(i + 1),
\ldots \pi^*\mathrm{D}^{b}(X) \otimes \cO(i + n - 1) \rangle.
\end{equation*}
\end{theorem}

\section{The Borel--Bott--Weil theorem}
\label{section:bbw}

The Borel--Bott--Weil theorem computes the cohomology of line bundles on the flag variety of a semisimple algebraic group. It can also be used to compute the cohomology of equivariant vector bundles on Grassmannians.  We restrict here to the cases of classical and isotropic Grassmannians.
\subsection{Classical Grassmannian}

Let $V$ be a vector space of dimension $n$.
We will use the standard identification of the weight lattice of the group $\mathrm{GL}(V)$ with $\mathbb{Z}^{n}$
that takes the fundamental weight of the representation $\Lambda^{k}V^{\vee}$ to the vector $(1, 1,\ldots, 1, 0, 0,\ldots , 0) \in \mathbb{Z}^{n}$
(the first $k$ entries are $1$, and the last $n-k$ are $0$).
We denote by
\begin{equation*}
\rho= (n, n- 1,\ldots, 2, 1)
\end{equation*}
the sum of fundamental weights of $\mathrm{GL}(V)$.

The cone of dominant weights of $\mathrm{GL}(V)$ gets identified with the set of non-increasing sequences
$\alpha = (a_{1}, a_{2},\ldots ,a_{n}$) of integers.
For such $\alpha$ we denote by $\Sigma^{\alpha}V^{\vee} = \Sigma^{a_{1},a_{2},\ldots,a_{n}}V^{\vee}$
the corresponding representation of $\mathrm{GL}(V)$ of highest weight $\alpha$.

Similarly, given a vector bundle $E$ of rank $n$ on a scheme $X$, we consider the corresponding principal $\mathrm{GL}(n)$-bundle on $X$ and denote by $\Sigma^{\alpha}E$ the vector bundle associated with the $\mathrm{GL}(n)$-representation of highest weight $\alpha$.

The Weyl group of $\mathrm{GL}(V)$ is isomorphic to the permutation group $\mathbf{S}_{n}$
and the length function $\ell \colon \mathbf{S}_n \to \mathbb{Z}$ counts the number of inversions in a permutation.
Note that for every weight $\alpha \in \mathbb{Z}^{n}$
there exists a permutation $\sigma \in \mathbf{S}_{n}$ such that $\sigma(\alpha)$ is dominant, i.e., non-increasing.


The linear algebraic group $\mathrm{GL}(V)$ acts naturally on the Grassmannian $\mathrm{Gr}(k,V)$ of $k$-dimensional subspaces in $V$.
Let~$\cU \subset V \otimes \mathcal{O}_{\mathrm{Gr}(k,V)}$ denote the tautological subbundle of rank $k$.
Denote by $V/\cU$ the corresponding quotient bundle and by $\cU^{\perp}$ its dual.
Every irreducible $\mathrm{GL}(V)$--equivariant vector bundle on $\mathrm{Gr}(k,V)$ is isomorphic to
\begin{equation*}
\Sigma^{\beta}\cU^{\vee}\otimes \Sigma^{\gamma}\cU^{\perp}
\end{equation*}
for some dominant weights $\beta \in \mathbb{Z}^{k}$ and $\gamma \in \mathbb{Z}^{n-k}$.

\begin{theorem}
Let $\beta \in \mathbb{Z}^{k}$ and $\gamma \in \mathbb{Z}^{n-k}$ be non-increasing sequences.
Let $\alpha =(\beta,\gamma) \in \mathbb{Z}^{n}$ be their concatenation.
Assume that all entries of $\alpha+\rho$ are distinct.
Let $\sigma \in \bS_n$ be the unique permutation such that $\sigma(\alpha+\rho)$ is strictly decreasing. Then
\begin{equation}
H^{p}(\mathrm{Gr}(k,V), \Sigma^{\beta}\cU^{\vee}\otimes \Sigma^{\gamma}\cU^{\perp} ) =
\Sigma^{\sigma(\alpha+\rho)-\rho}\,V^{\vee}[-\ell(\sigma)].
\end{equation}
If not all entries of $\alpha+\rho$ are distinct then
\begin{equation*}
H^{\bullet}(\mathrm{Gr}(k,V), \Sigma^{\beta}\cU^{\vee}\otimes \Sigma^{\gamma}\cU^{\perp} )=0.
\end{equation*}
\end{theorem}

There is a consequence of the general Borel--Bott--Weil theorem,
that computes direct images under the natural projection of relative Grassmannian $q\colon \mathrm{Gr}_{\mathrm{Gr}(l,V)}(k,E)\to \mathrm{Gr}(l,V)$,
where $E$ is a vector bundle on~$\mathrm{Gr}(l,V)$.
Let us denote by $\cU_{k}$ the taulological bundle on $\mathrm{Gr}_{\mathrm{Gr}(l,V)}(k,E)$.
The following proposition describes the direct images of some bundles of the form~$\Sigma^{\beta}\cU_{k}^{\vee}$:

\begin{proposition} \label{45}
Let $\beta=(\beta_1,\ldots,\beta_k)\in \mathbb{Z}^{k}$ be a non-increasing sequence of integers. Denote by~$\alpha= (\beta;0,\ldots,0) {} \in \mathbb{Z}^l$. Assume that all entries of $\alpha + \rho$ are distinct.
Let $\sigma\in \mathbf{S}_{l}$ be the unique permutation such that $\sigma(\alpha + \rho)$ is strictly decreasing.
Then
\begin{equation}
Rq_*(\Sigma^{\beta}\mathcal{U}_k^{\vee}) =
\Sigma^{\sigma(\alpha + \rho)-\rho}\,\cU_{l}^{\vee}[-\ell(\sigma)].
\end{equation}
If not all entries of $\alpha$ are distinct then $Rq_*(\Sigma^{\beta}\cU_k^{\vee})=0$.
\end{proposition}

Also we will use the Littlewood–Richardson rule, that provides a recipe
to decompose the tensor product $\Sigma^{\alpha}\cU \otimes \Sigma^{\beta}\cU^{\vee}$
into a direct sum of bundles of the form $\Sigma^{\gamma}\cU^{\vee}$.
We refer to \cite{5} for the precise formulation of this rule.
Let us just mention that we have the following property of this decomposition.

\begin{lemma} \label{lemma:lrr}
Let $\alpha=(\alpha_1, \ldots,\alpha_k)\in \mathbb{Z}^{k}$ and $\beta=(\beta_1, \ldots, \beta_k)\in \mathbb{Z}^{k}$ be non-increasing sequences.
Then there is a direct sum decomposition
\begin{equation*}
\Sigma^{\alpha_1, \ldots,\alpha_k }\cU \otimes \Sigma^{\beta_1, \ldots, \beta_k}\cU^{\vee}=
\bigoplus \Sigma^{\gamma_1,\ldots, \gamma_k}\cU^{\vee},
\end{equation*}
where
\begin{equation*}
- \alpha_{k+1-i} \le \gamma_i \le \beta_i\quad\text{for all $1 \le i \le k$}\qquad\text{and}\qquad
\displaystyle\sum \gamma_i = \sum \beta_i - \sum \alpha_i.
\end{equation*}
\end{lemma}

\subsection{Isotropic Grassmannian}

Now, let $V$ be a $2n$-dimensional vector space with a fixed symplectic form.
The weight lattice of the corresponding symplectic group $\mathrm{Sp}(V)$ can be identified with $\mathbb{Z}^{n}$:
under this identification, as before, the $k$-th fundamental weight goes to $(1,\dots,1,0,\dots,0)$
(the first $k$ entries are $1$, and the last $n-k$ are $0$).
Denote by
\begin{equation*}
\rho= (n, n-1,\ldots, 2, 1)
\end{equation*}
the sum of the fundamental weights of $\mathrm{Sp}(V)$.

The cone of dominant weights of $\mathrm{Sp}(V)$ gets identified with the set of non-increasing sequences $\alpha = (a_{1}, a_{2},\ldots ,a_{n}$)
of non-negative integers.
For such $\alpha$ we denote by $\Sigma^{\alpha}_{\mathrm{Sp}}V = \Sigma^{a_{1},a_{2},\ldots,a_{n}}_{\mathrm{Sp}}V$
the corresponding representation of $\mathrm{Sp}(V)$.

Similarly, given a symplectic vector bundle $E$ of rank $2n$ on a scheme $X$,
we consider the corresponding principal $\mathrm{Sp}(2n)$-bundle on $X$ and denote by $\Sigma^{\alpha}_{\mathrm{Sp}}E$
the vector bundle associated with the~$\mathrm{Sp}(2n)$-representation of highest weight $\alpha$.

The Weyl group of $\mathrm{Sp}(V)$ is equal to a semidirect product of $\mathbf{S}_{n}$ and $(\mathbb{Z}/2\mathbb{Z})^{n},$
where $\mathbf{S}_{n}$ acts on the weight lattice $\mathbb{Z}^n$ by permutations and $(\mathbb{Z}/2\mathbb{Z})^{n}$
acts by changes of signs of the coordinates. Let us denote by $\ell \colon \mathbf{S}_{n}\ltimes(\mathbb{Z}/2\mathbb{Z})^{n}\to \mathbb{Z}$
the corresponding length function.
Note that for every $\alpha \in \mathbb{Z}^{n}$ there exists an element $\sigma$ of the Weyl group such that $\sigma(\alpha)$
is dominant, i.e., is a non-increasing sequence of non-negative integers.

The symplectic group $\mathrm{Sp}(V)$ acts naturally on the Grassmannian $\mathrm{IGr}(k,V) \subset \Gr(k,V)$
of isotropic $k$-dimensional subspaces.
Let $\mathcal{U} \subset V \otimes \mathcal{O}_{\mathrm{IGr}(k, V)}$ denote the tautological subbundle of rank $k$.
Denote by~$V/\mathcal{U}$ the corresponding quotient bundle and by $\cU^{\perp}$ its dual.
Note that all these bundles are restricted from $\Gr(k,V)$.
The main difference is that
in the case of isotropic Grassmannian the tautological bundle~$\mathcal{U}$ is naturally a subbundle of $\mathcal{U}^{\perp},$
i.e. $\mathcal{U} \subset \mathcal{U}^{\perp},$ so we also have the quotient bundle
\begin{equation*}
\cS := \mathcal{U}^{\perp}/\mathcal{U}
\end{equation*}
of rank $2n-2k.$
Note that $\cS$ is a symplectic bundle.
Every irreducible $\mathrm{Sp}(V)$--equivariant vector bundle on $\IGr(k,V)$ is isomorphic to
\begin{equation*}
\Sigma^{\beta}\cU^{\vee}\otimes \Sigma^{\gamma}_{\mathrm{Sp}}\cS
\end{equation*}
for some $\GL$-dominant weight $\beta \in \mathbb{Z}^{k}$ and $\Sp$-dominant weight $\gamma \in \mathbb{Z}^{n-k}$.

\begin{theorem} \label{12}
Let $\beta \in \mathbb{Z}^{k}$ be a non-increasing sequence and
let $\gamma \in \mathbb{Z}^{n-k}$ be a non-increasing sequence of non-negative integers.
Let $(\alpha_1,\ldots,\alpha_n)=\alpha =(\beta,\gamma) \in \mathbb{Z}^{n}$ be their concatenation.
Assume that all entries of $\alpha+\rho$ are non-zero integers with distinct absolute values.
Let $\sigma$ be the unique element of the Weyl group~$\mathbf{S}_{n}\ltimes(\mathbb{Z}/2\mathbb{Z})^{n}$
such that $\sigma(\alpha+\rho)$ is a strictly decreasing sequence of positive integers.
Then
\begin{equation}
H^{p}(\mathrm{IGr}(k,V), \Sigma^{\beta}\mathcal{U}^{\vee}\otimes \Sigma_{\mathrm{Sp}}^{\gamma}\cS) =
\Sigma_{\mathrm{Sp}}^{\sigma(\alpha+\rho)-\rho}\,V^{\vee}[-\ell(\sigma)].
\end{equation}
If not all entries of $\alpha+\rho$ have distinct non-zero absolute values then
\begin{equation*}
H^{\bullet}(\mathrm{IGr}(k,V), \Sigma^{\beta}\mathcal{U}^{\vee}\otimes \Sigma^{\gamma}_{\mathrm{Sp}}\cS)=0.
\end{equation*}
\end{theorem}

The following simple consequence of the Borel--Bott--Weil theorem is quite useful.

\begin{corollary} \label{11}
In the notation of the previous theorem suppose there exists $m\in \mathbb{Z},$
such that
\begin{equation*}
\#\{i \text{ such that }  |\alpha_i+\rho_i|\le m \}\ge m+1.
\end{equation*}
Then for all $p$ we have $$H^{p}(\mathrm{IGr}(k,V), \Sigma^{\beta}\mathcal{U}^{\vee}\otimes \Sigma_{\mathrm{Sp}}^{\gamma}\cS)=0.$$
\end{corollary}

\section{Some cohomology computations and the rectangular part of $\mathrm{D}^{b}(\mathrm{IGr}(3,8))$}
\label{section:db}

From now on we denote by $V$ an $8$-dimensional vector space with a fixed symplectic form $\omega$.
In what follows we will always identify $V$ with $V^{\vee}$ via the symplectic form $\omega$.
We denote
\begin{equation*}
X = \mathrm{IGr}(3,V)
\end{equation*}
and note that $\dim X = 12$ and
\begin{equation}
\label{eq:kx}
\omega_X \cong \cO(-6),
\end{equation}
where $\cO(1) = \det(\cU^\vee)$ is the ample generator of the Picard group $\operatorname{Pic}(X)$.
Note also that the rank of the Grothendieck group of $X$, equal to the index of the subgroup
$\bS_3 \times (\ZZ/2\ZZ) \subset \bS_4 \rtimes (\ZZ/2\ZZ)^4$, is $32$.

Recall that any irreducible $\mathrm{Sp}(V)$-equivariant bundle on $\mathrm{IGr}(3,V)$ is isomorphic to
$\Sigma^{\beta}\mathcal{U}^{\vee}\otimes \Sigma_{\mathrm{Sp}}^{\gamma}\cS$,
where now $\cU$ is the tautological bundle of rank $3$ and $\cS := \cU^\perp/\cU$ is a symplectic bundle of rank $2$,
so~$\beta\in \mathbb{Z}^{3}$ and~$\gamma\in \mathbb{Z}$.
In particular, the corresponding weight $\gamma$ is just a single non-negative integer~$\gamma =(c)$, and
\begin{equation*}
\Sigma^{\gamma}_{\mathrm{Sp}}\cS\simeq S^{c}\cS
\end{equation*}
since $\mathrm{Sp}_2\simeq \mathrm{SL}_2$.
Thus, any irreducible $\mathrm{Sp}(V)$-equivariant bundle on $\mathrm{IGr}(3,V)$ can be written as
\begin{equation*}
\Sigma^{\beta}\mathcal{U}^{\vee}\otimes S^c.
\end{equation*}
We will frequently use the following natural identifications:
\begin{equation*}
\Sigma^{\alpha_1,\alpha_2,\alpha_3}\cU \cong \Sigma^{-\alpha_3,-\alpha_2,-\alpha_1}\cU^\vee,
\qquad
\Sigma^{\alpha_1+t,\alpha_2+t,\alpha_3+t}\cU^\vee \cong \Sigma^{\alpha_1,\alpha_2,\alpha_3}\cU^\vee \otimes \cO(t).
\end{equation*}

In what follows for a dominant weight of the form $(\alpha_1,\alpha_2,0)$ we will omit the last zero,
i.e., we will write just~$(\alpha_1,\alpha_2)$~and~$\Sigma^{\alpha_1,\alpha_2}\cU^{\vee}$ for such a weight.
Note also that $\Sigma^{a,0,0}\cU^\vee \cong S^a\cU^\vee$ and $\Sigma^{1,1}\cU^\vee = \Lambda^2\cU^\vee$.


\subsection{The rectangular part}

We consider the partial ordering on dominant weights of $\mathrm{GL}_3$ defined by:
\begin{align*}
(\alpha_1,\alpha_2,\alpha_3) &\le (\beta_1,\beta_2,\beta_3)  &&
\text{if $\alpha_1 \le \beta_1$ and $\alpha_2 \le \beta_2$ and $\alpha_3 \le \beta_3$},
\\
(\alpha_1,\alpha_2,\alpha_3) &< (\beta_1,\beta_2,\beta_3)  &&
\text{if $\alpha \le \beta$ and $\alpha \ne \beta$}.
\end{align*}
The next computation allows to construct the rectangular part of the desired Lefschetz collection.

\begin{lemma}
\label{19}
For $(0,0) \le \alpha \le(2,1)$, $(0,0) \le \beta \le (2,1)$, and $0 \le k \le 5$ we have
\begin{equation*}
\mathrm{Ext}^{\bullet}(\Sigma^{\alpha}\cU^{\vee}(k),\Sigma^{\beta}\cU^{\vee})\ne 0 {}\iff{} \alpha\le \beta\  \text{and}\ k=0 .
\end{equation*}
Moreover, for $k = 0$ we have
\begin{equation*}
\mathrm{Hom}(\Sigma^{\alpha}\cU^{\vee},\Sigma^{\beta}\cU^{\vee})=\begin{cases}
\Bbbk,&\text{if $\alpha=\beta$};\\
\Sigma^{\beta}_{\mathrm{Sp}}V, &\text{if $\alpha=0$};\\
V, &\text{if $\alpha=(1,0,0)$ and $\beta=(1,1,0)$ or $\beta=(2,0,0)$};\\
(V\otimes V)/\omega, &\text{if $\alpha=(1,0,0)$ and $\beta=(2,1,0)$};\\
V, &\text{if $\alpha=(1,1,0)$ or $\alpha=(2,0,0),$ and $\beta=(2,1,0)$}.
\end{cases}
\end{equation*}
All other $\mathrm{Ext}^{\bullet}(\Sigma^{\alpha}\cU^{\vee},\Sigma^{\beta}\cU^{\vee})$ are equal to zero.
\end{lemma}
\begin{proof}
We need to compute
$\mathrm{Ext}^{\bullet}(\Sigma^{\alpha}\cU^{\vee}(k),\Sigma^{\beta}\cU^{\vee})=
\mathrm{H}^{\bullet}(X,\Sigma^{\alpha}\cU\otimes \Sigma^{\beta}\cU^{\vee}(-k))$,
where $(0,0) \le \alpha \le(2,1)$ and~$(0,0) \le \beta \le (2,1)$.
Using the Littlewood–Richardson rule we decompose
\begin{equation*}
\Sigma^{\alpha}\cU\otimes \Sigma^{\beta}\cU^{\vee}(-k)= \bigoplus \Sigma^{\gamma}\cU^{\vee}(-k),
\end{equation*}
where $\gamma=(\gamma_1,\gamma_2,\gamma_3)$ and $\gamma_1\in [0,2],$ $\gamma_2 \in [-1,1],$ $\gamma_3\in [-2,0]$
by Lemma~\ref{lemma:lrr}, and so we need to compute~$\mathrm{H}^{\bullet}(X,\Sigma^{\gamma}\cU^{\vee}(-k))$.

Let us denote by $(x_1,x_2,x_3,x_4)=(\gamma_1-k,\gamma_2-k,\gamma_3-k,0)+\rho=(\gamma_1+4-k,\gamma_2+3-k,\gamma_3+2-k,1)$.
So we get that $x_1\in [4-k,6-k],$ $x_2 \in [2-k,4-k],$ $x_3=[-k,2-k]$ and $x_4=1$.
In each of the next five cases we apply Corollary~\ref{11} to conclude that the cohomology groups vanish.
\begin{itemize}
\item
If $k=1$ then $|x_3|\le 1$ and $x_4=1$ so
all cohomology groups vanish.
\item
If $k=2$ then $|x_3|\le 2,$ $|x_2|\le 2$ and $x_4=1$ so
all cohomology groups vanish.
\item
If $k=3$ then $|x_3|\le 3,$ $|x_2|\le 3$, $|x_1|\le 3$ and $x_4=1$ so
all cohomology groups vanish.
\item
If $k=4$ then $|x_1|\le 2,$ $|x_2|\le 2$ and $x_4=1$ so
all cohomology groups vanish.
\item
If $k=5$ then $|x_1|\le 1$ and $x_4=1$ so
all cohomology groups vanish.
\end{itemize}
Now, finally, consider the case $k=0$.
Assume that not all the cohomology groups vanish.
Since $x_4=1$ and~$x_3 \in [0,2]$ it follows that $x_3=2>x_4$.
Since $x_2\in [2,4]$ it follows that $x_2 \ge 3 > x_3$.
Since~$x_1\in [4,6]$ and $x_1 {}\ge{} x_2$ we conclude that $\gamma+\rho$ is strictly dominant, hence $\gamma$ is dominant,
hence the only non-trivial cohomology group is $\mathrm{H}^0$.
Furthermore, by the Littlewood--Richardson rule we deduce that $\alpha\le \beta$.
Finally, considering the summands with dominant $\gamma$ in $\Sigma^{\alpha}\cU\otimes \Sigma^{\beta}\cU^{\vee}$
we deduce the required formula for the $\operatorname{Hom}$-spaces.
\end{proof}

Recall the collection~\eqref{eq:e} of five vector bundles on $\IGr(3,V)$ defined in the Introduction.

\begin{corollary}
\label{corollary:e}
The following collection of $30$ vector bundles
\begin{equation*}
\mathfrak{E},\mathfrak{E}(1),\mathfrak{E}(2),\mathfrak{E}(3),\mathfrak{E}(4),\mathfrak{E}(5)
\end{equation*}
is a rectangular Lefschetz exceptional collection in $\mathrm{D}^b(\mathrm{IGr}(3,V))$ with respect to $\cO(1)$.
\end{corollary}

Recall also the collection $\mathfrak{E}'$ defined in~\eqref{eq:e-prime}.
Using Proposition~\ref{proposition:long-mutations} and~\eqref{eq:kx} we deduce

\begin{corollary}
\label{corollary:e-prime}
The following collection of $30$ vector bundles
\begin{equation*}
\mathfrak{E}',\mathfrak{E}'(1),\mathfrak{E}'(2),\mathfrak{E}'(3),\mathfrak{E}'(4),\mathfrak{E}'(5)
\end{equation*}
is a rectangular Lefschetz exceptional collection in $\mathrm{D}^b(\mathrm{IGr}(3,V))$ with respect to $\cO(1)$.
\end{corollary}

As the rank of the Grothendieck group is 32, the above collections are not full,
and to complete them we need to add two more objects.
We will do this in Section~\ref{section:non-rectangular} after some preparations.

\subsection{An extra bundle}

Now we will try to add the vector bundle $\Sigma^{3,1}\cU^\vee \cong (S^3\cU^\vee \otimes \cU^\vee) / S^4\cU^\vee$
to the exceptional collection.
Actually, it does not fit, but as we will see later its modification does,
so the results from this section will be useful.

\begin{lemma} \label{31U}
The vector bundle $\Sigma^{3,1}\cU^{\vee}$ is exceptional.
\end{lemma}
\begin{proof}
By the Littlewood–Richardson rule we have
\begin{equation*}
\mathrm{Ext}^{\bullet}(\Sigma^{3,1}\cU^{\vee},\Sigma^{3,1}\cU^{\vee}) =
\mathrm{H}^{\bullet}(X,\Sigma^{3,1}\cU\otimes \Sigma^{3,1}\cU^{\vee}) =
\bigoplus \mathrm{H}^{\bullet}(X,\Sigma^{\gamma}\cU^{\vee}),
\end{equation*}
where for each $\gamma = (\gamma_1,\gamma_2,\gamma_3)$ we have $\gamma_3\in[-3,0]$ by Lemma~\ref{lemma:lrr}.
Using Corollary~\ref{11} we see that a summand $\Sigma^\gamma\cU^\vee$ can have non-zero cohomology groups only if $\gamma_3=0$.
But since $\gamma_1\ge\gamma_2\ge\gamma_3$ and~$\gamma_3+\gamma_2+\gamma_1=0$ (again by Lemma~\ref{lemma:lrr}), we conclude that
the only summand having non-zero cohomology groups is $\Sigma^{\gamma}\cU^{\vee}$ with~$\gamma=(0,0,0)$,
and hence $\Sigma^{3,1}\cU^{\vee}$ is exceptional.
\end{proof}

\begin{lemma}
\label{18}
For $(0,0) \le \alpha \le(2,1)$ and $1 \le k \le 6$, we have
\begin{equation*}
\mathrm{Ext}^{\bullet}(\Sigma^{\alpha}\cU^{\vee}(k),\Sigma^{3,1}\cU^{\vee})=\begin{cases}
\Bbbk[-9], &\text{if $\alpha=(2,1)$ and $k=5$};\\
0, &\text{in all other cases}.
\end{cases}
\end{equation*}
\end{lemma}
\begin{proof}
The computation is similar to the one in the proof of Lemma~\ref{19}.
First of all, we have
$\mathrm{Ext}^{\bullet}(\Sigma^{\alpha}\cU^{\vee}(k),\Sigma^{3,1}\cU^{\vee})=
\mathrm{H}^{\bullet}(X,\Sigma^{\alpha}\cU \otimes \Sigma^{3,1}\cU^\vee(-k))$.
By the Littlewood--Richardson rule we see that
\begin{equation*}
\Sigma^{\alpha}\cU\otimes \Sigma^{3,1}\cU^{\vee}(-k)= \bigoplus \Sigma^{\gamma}\cU^{\vee}(-k)
\end{equation*}
and we need to compute $\mathrm{H}^{\bullet}(X,\Sigma^{\gamma}\cU^{\vee}(-k))$.
Here $\gamma=(\gamma_1,\gamma_2,\gamma_3)$ and by Lemma~\ref{lemma:lrr}
we have $\gamma_1\in [0,3]$ $\gamma_2 \in [-1,1],$ $\gamma_3\in [-2,0]$, and~$\gamma_1+\gamma_2+\gamma_3 \in [1,4]$.
A more careful analysis of the Littlewood--Richardson rule shows that $\gamma_1 \ge 1$, so that $\gamma_1 \in [1,3]$.

Let us denote $(x_1,x_2,x_3,x_4)=(\gamma_1-k,\gamma_2-k,\gamma_3-k,0)+\rho=(\gamma_1+4-k,\gamma_2+3-k,\gamma_3+2-k,1)$.
So we get that $x_1\in [5-k,7-k],$ $x_2 \in [2-k,4-k],$ $x_3=[-k,2-k],$ $x_4=1$ and $x_1+x_2+x_3+x_4\in [11-3k,14-3k]$.
In each of the next six cases we apply Corollary~\ref{11} to conclude that the cohomology vanishes.
\begin{itemize}
\item
If $k=1$ then $|x_3|\le 1$ and $x_4=1$ so
all cohomology groups vanish.
\item
If $k=2$ then $|x_3|\le 2,$ $|x_2|\le 2$ and $x_4=1$ so
all cohomology groups vanish.
\item
If $k=3$ then $|x_2|\le 1$ and $x_4=1$ so
all cohomology groups vanish.
\item
For $k=4$ we have $|x_2|\le 2$ and $x_4=1,$ so if $|x_1|\le 2$ or $|x_2| \le 1$ then
all cohomology groups vanish.
The only remaining case is when
$x_1=3,$ $x_2=-2$.
But then $x_4+x_3+x_2+x_1=-2$ and this is a contradiction with the condition $x_4+x_3+x_2+x_1\in [-1,2].$
\item
For $k=5$ we see that if $|x_2|\le 2$ or $x_1 \le 1$ or $x_3=-3$ then
all cohomology groups vanish.
If~$x_1=2,$ $x_2=-3$ and $x_3=-5$ then $x_4+x_3+x_2+x_1=-5$ and this is a contradiction with the condition $x_4+x_3+x_2+x_1\in [-4,-1]$.
If $x_1=2,$ $x_2=-3$ and~$x_3=-4$ then for the corresponding bundle $\Sigma^{3,-1,-1}\cU^{\vee}(-5)$
we have $\mathrm{H}^{9}(\Sigma^{3,-1,-1}\cU^{\vee}(-5))=\Bbbk$ and since~$\Sigma^{3,-1,-1}\cU^{\vee}(-5)$
has multiplicity one in the direct sum decomposition of~$\Sigma^{2,1}\cU \otimes \Sigma^{3,1}\cU^{\vee}(-5)$
we conclude that~$\mathrm{Ext}^{9}(\Sigma^{2,1}\cU^{\vee}(5),\Sigma^{3,1}\cU^{\vee})=\Bbbk$.
\item
If $k=6$ then $|x_1|\le 1$ and $x_4=1$ so
all cohomology groups vanish.
\end{itemize}
This completes the proof of the lemma.
\end{proof}

Using Serre duality and~\eqref{eq:kx}, we deduce

\begin{corollary} \label{17}
For $(0,0) \le \alpha \le(2,1)$ and $0 \le k \le 5$, we have
\begin{equation*}
\mathrm{Ext}^{\bullet}(\Sigma^{3,1}\cU^{\vee}(k), \Sigma^{\alpha}\cU^{\vee}) \cong
\begin{cases}
\Bbbk[-3], &\text{if $\alpha=(2,1)$ and $k = 1$};\\
0, & \text{otherwise}.
\end{cases}
\end{equation*}
\end{corollary}


We will also need the following lemmas.
\begin{lemma} \label{all31}
For $(0,0) \le \alpha \le(2,1)$ we have
\begin{equation*}
\mathrm{Hom}(\Sigma^{\alpha}\cU^{\vee},\Sigma^{3,1}\cU^{\vee})=\begin{cases}
\Sigma^{3,1}V/S^2V, &\text{if $\alpha=0$};\\
(S^2V\otimes V)/V , &\text{if $\alpha=(1,0,0)$};\\
S^2V, &\text{if $\alpha=(1,1,0)$};\\
(V\otimes V)/\omega, &\text{if $\alpha=(2,0,0)$};\\
V, &\text{if $\alpha=(2,1,0)$}.
\end{cases}
\end{equation*}
All other $\mathrm{Ext}^{\bullet}(\Sigma^{\alpha}\cU^{\vee},\Sigma^{3,1}\cU^{\vee})$ are equal to zero.
\end{lemma}

\begin{proof}
The statement of the lemma follows from the Littlewood--Richardson rule and the Borel--Bott--Weil theorem.
\end{proof}
\begin{lemma} \label{31(3)}
We have
\begin{equation*}
\mathrm{Ext}^{\bullet}(\Sigma^{3,1}\cU^{\vee}(3), \Sigma^{3,1}\cU^{\vee}) = 0.
\end{equation*}
\end{lemma}
\begin{proof}
Using the Littlewood--Richardson rule we have
\begin{equation*}
\mathrm{Ext}^{\bullet}(\Sigma^{3,1}\cU^{\vee}(3), \Sigma^{3,1}\cU^{\vee}) =
\bigoplus \mathrm{H}^{\bullet}(X,\Sigma^{\gamma}\cU^{\vee}),
\end{equation*}
where for $\gamma=(\gamma_1,\gamma_2,\gamma_3)$ we have $\gamma_2\in[-4,-2]$ by Lemma~\ref{lemma:lrr}.
So for $(x_1,x_2,x_3,x_4) = \gamma + \rho$ we have~$x_2\in[-1,1]$ and~$x_4 = 1$.
Hence using Corollary~\ref{11} we deduce that all cohomology groups vanish.
\end{proof}

\subsection{Some semiorthogonalities}


Here we establish some semiorthogonalities that will be useful later.

\begin{lemma} \label{lemma:32}
The bundles $\cU^\vee$, $\cO$, $\cO(-1)$ and $\Lambda^2\cU^\vee(-2)$ are left orthogonal to $\Sigma^{3,2}\cU^\vee(-3)$.
Moreover, we have $\mathrm{Ext}^{\bullet}(S^2\cU^{\vee},\Sigma^{3,2}\cU^\vee(-3))=\Bbbk[-4]$.
\end{lemma}
\begin{proof}
Using the Littlewood--Richardson rule and Corollary~\ref{11} we have
\begin{equation*}
\operatorname{Ext}^\bullet(\cO(-k),\Sigma^{3,2}\cU^\vee(-3))=\mathrm{H}^\bullet(X,\Sigma^{0,-1,-3}\cU^\vee(k))=0
\end{equation*}
for $k\in\{0,1\}$, which gives the required semiorthogonality for $\cO$ and $\cO(-1)$.
Similarly,
we have
\begin{equation*}
\operatorname{Ext}^\bullet(\cU^{\vee},\Sigma^{3,2}\cU^\vee(-3))=
\mathrm{H}^\bullet(X,\Sigma^{-1,-1,-3}\cU^{\vee}\oplus \Sigma^{0,-2,-3}\cU^{\vee}\oplus \Sigma^{0,-1,-4}\cU^{\vee})=0,
\end{equation*}
which gives the required semiorthogonality for $\cU^\vee$.
Finally,
we have
\begin{equation*}
\operatorname{Ext}^\bullet(\Lambda^2\cU^\vee(-2),\Sigma^{3,2}\cU^\vee(-3))=
\mathrm{H}^\bullet(X,\Sigma^{1,0,-1}\cU^{\vee}\oplus \Sigma^{2,0,-2}\cU^{\vee}\oplus \Sigma^{1,1,-2}\cU^{\vee})=0.
\end{equation*}
which gives the semiorthogonality for $\Lambda^2\cU^\vee(-2)$ and
completes the proof of the first part of the lemma.

For the second part we use again the Littlewood--Richardson rule and Corollary~\ref{11} to conclude
\begin{multline*}
\mathrm{Ext}^{\bullet}(S^2\cU^{\vee},\Sigma^{3,2}\cU^\vee(-3))= \\
=\mathrm{H}^{\bullet}(X, \Sigma^{0,-1,-5}\cU^{\vee} \oplus \Sigma^{0,-3,-3}\cU^{\vee}\oplus\Sigma^{0,-2,-4}\cU^{\vee}\oplus\Sigma^{-2,-2,-2}\cU^{\vee}
\oplus\Sigma^{-1,-1,-4}\cU^{\vee}\oplus\Sigma^{-1,-2,-3}\cU^{\vee}) = \\ =
\Bbbk[-4],
\end{multline*}
where the nontrivial contribution comes from the first summand.
\end{proof}

\begin{lemma} \label{T1ort}
The vector bundles $\Sigma^{3,3}\cU^{\vee}(-4),$ $\Sigma^{2,2}\cU^{\vee}(-3)$ and $\Sigma^{3,2}\cU^{\vee}(-3)$
are right orthogonal to the exceptional collection~$\langle \mathcal{O},\cU^{\vee},\Lambda^{2}\cU^{\vee},\Sigma^{2,1}\cU^{\vee}\rangle.$
\end{lemma}
\begin{proof}
We want to prove that $\mathrm{Ext}^{\bullet}(\mathcal{E}_1,\mathcal{E}_2)=0,$ where $\mathcal{E}_1$ is one of the bundles
from the exceptional collection $(\mathcal{O},\cU^{\vee},\Lambda^{2}\cU^{\vee},\Sigma^{2,1}\cU^{\vee})$
and $\mathcal{E}_2$ is either $\Sigma^{3,3}\cU^{\vee}(-4)$ or $\Sigma^{2,2}\cU^{\vee}(-3)$ or $\Sigma^{3,2}\cU^{\vee}(-3)$.

We have
\begin{equation*}
\mathrm{Ext}^{\bullet}(\mathcal{E}_1,\mathcal{E}_2)
=\mathrm{H}^{\bullet}(X,\mathcal{E}_1^{\vee}\otimes \mathcal{E}_2)
=\bigoplus \mathrm{H}^{\bullet}(X,\Sigma^{\gamma}\cU^{\vee}),
\end{equation*}
where for $\gamma=(\gamma_1,\gamma_2, \gamma_3)$ we have $\gamma_1\in[-2,0],$  $\gamma_2\in[-3,-1]$ and $\gamma_3\in[-6,-3]$
by Lemma~\ref{lemma:lrr}.
Let us denote $x=(x_1,x_2,x_3,x_4) = \gamma+\rho$.
Then we have $x_1\in[2,4],$ $x_2\in [0,2]$, $x_3\in [-4,-1]$, and $x_4 = 1$.
If~$\mathrm{H}^{\bullet}(X,\Sigma^{\gamma}\cU^{\vee})\ne 0,$ then using Corollary \ref{11} we deduce
that $x_2=2,$ $x_1\in[3,4]$ and $x_3\in[-4,-3]$.
More precisely, either $x=(4,2,-3,1)$ or $x=(3,2,-4,1)$, i.e., either $\gamma=(0,-1,-5)$ or $\gamma=(-1,-1,-6)$
and using the Littlewood–Richardson rule we see that the corresponding $\Sigma^{\gamma}\cU^{\vee}$
could only come from the tensor product $S^2\cU\otimes\mathcal{E}_{2}$.
\end{proof}

Recall that we denote by $\cS$ the vector bundle $\cU^\perp/\cU$.

\begin{lemma} \label{21(-1)}
The bundle $\Sigma^{2,2}\cU^{\vee}(-3)\otimes \cS$ is right orthogonal to $\Sigma^{2,1}\cU^\vee(-1)$.
\end{lemma}
\begin{proof}
We need to compute $\mathrm{Ext}^{\bullet}(\Sigma^{2,1}\cU^{\vee}(-1),\Sigma^{2,2}\cU^{\vee}(-3)\otimes \cS)=
\mathrm{H}^{\bullet}(X,\Sigma^{2,1}\cU^{\vee}\otimes\Sigma^{2,2}\cU^{\vee}(-4)\otimes \cS)$.
Using the Littlewood–Richardson rule we see that
\begin{equation*}
\mathrm{H}^{\bullet}(X, \Sigma^{2,1}\cU^{\vee}\otimes\Sigma^{2,2}\cU^{\vee}(-4)\otimes \cS)
= \bigoplus \mathrm{H}^{\bullet}(X,\Sigma^{\gamma}\cU^{\vee}\otimes\cS),
\end{equation*}
where for $\gamma =(\gamma_1,\gamma_2, \gamma_3)$ we have $\gamma_1\in[-1,0]$,  $\gamma_2\in[-2,-1]$ and $\gamma_3\in[-4,-2]$
by Lemma~\ref{lemma:lrr}.
Let us denote by $x=(x_1,x_2,x_3,x_4)=(\gamma_1,\gamma_2, \gamma_3,1)+\rho$.
Then we have $x_2\in [1,2],$ $x_3\in [-2,0]$ and $x_4=2$, so using Corollary \ref{11} we deduce the statement.
\end{proof}

\section{The bicomplex}
\label{section:exact}

In this section we construct the bicomplex described in the Introduction
and discuss its properties.

\subsection{Koszul and staircase complexes}

We start with some well-known exact sequences.
The (dual) tautological exact sequence
\begin{equation*}
0\to \cU^{\perp}\to V\otimes \cO \to \cU^{\vee}\to 0
\end{equation*}
on $\mathrm{IGr}(3,V)$ induces for each $k$ the following long exact sequence (Koszul complex)
\begin{equation}
\label{15}
0\to \Lambda^k\cU^{\perp}\to \Lambda^kV\otimes \cO \to \Lambda^{k-1}V\otimes\cU^{\vee} \to \ldots
\to V\otimes S^{k-1}\cU^{\vee} \to S^k\cU^{\vee} \to 0.
\end{equation}
All these exact sequences are restricted from $\mathrm{Gr}(3,V)$.

We will also need so called \emph{staircase complexes} (see \cite{4}).
We will not describe the general form of these complexes here, but list those that we will use (recall that we identify $V$ and $V^{\vee}$ via $\omega$):
\begin{itemize}
\item The staircase complex for $S^2\cU^{\vee}$:
\begin{multline} \label{29}
0 \to
\Sigma^{3,3}\cU^{\vee}(-4) \to
\Lambda^7V\otimes \Sigma^{2,2}\cU^{\vee}(-3) \to
\Lambda^6V\otimes \Lambda^2\cU^{\vee}(-2) \to
\Lambda^5V \otimes \mathcal{O}(-1) \to
\\ \to
\Lambda^2V \otimes \cO
\to V\otimes\cU^{\vee} \to
S^2\cU^{\vee} \to 0.
\end{multline}
Its second line coincides with the Koszul complex~\eqref{15} for $k = 2$, and its first line is a twist of the dual of~\eqref{15} for $k = 3$.
\item The staircase complex for $\Sigma^{3,1}\cU^{\vee}$:
\begin{multline}  \label{34}
0 \to
\Sigma^{3,2}\cU^{\vee}(-3) \to
\Lambda^7V\otimes \Sigma^{2,1}\cU^{\vee}(-2) \to
\Lambda^6V\otimes \cU^{\vee}(-1) \to
\\ \to
\Lambda^4V \otimes \mathcal{O} \to
\Lambda^2V \otimes \Lambda^{2}\cU^{\vee}
\to V\otimes \Sigma^{2,1}\cU^{\vee}
\to \Sigma^{3,1}\cU^{\vee} \to 0.
\end{multline}
This complex is self-dual.
\item The staircase complex for $\Sigma^{3,2}\cU^{\vee}$:
\begin{multline}
\label{staircase}
0\to \Sigma^{4,2}\cU^{\vee}(-3)\to \Lambda^7V\otimes \Sigma^{3,1}\cU^{\vee}(-2)\to \Lambda^6V\otimes S^2\cU^{\vee}(-1) \to\\
 \to \Lambda^4V\otimes \cU^{\vee}\to \Lambda^3V\otimes \Lambda^2\cU^{\vee}\to V\otimes \Sigma^{2,2}\cU^{\vee}\to \Sigma^{3,2}\cU^{\vee}\to 0.
\end{multline}
\end{itemize}
All these complexes are also restricted from $\mathrm{Gr}(3,V)$.

Finally, for the rank-2 symplectic bundle $\cS {} = \cU^\perp/\cU$ we have the following exact sequences:
\begin{equation*}
0\to \cU \to \mathcal{U}^{\perp} \to \cS \to 0
\qquad\text{and}\qquad
0\to \cS \to V/\cU \to \cU^{\vee} \to 0.
\end{equation*}
Moreover, we have an isomorphism in $\mathrm{D}^b(X)$
\begin{equation}
\label{eq:upu-monad}
\cS \cong \Big\{ \cU \to V \otimes \cO \to \cU^\vee \Big\}.
\end{equation}

\subsection{Bicomplex}

The goal of this section is to construct a morphism of complexes from~\eqref{29} to~\eqref{34}
that we will consider as a bicomplex.
In the construction we will use the following lemma.

\begin{lemma}
\label{lemma:extension}
Let $E^\bullet$, and $F^\bullet$ be a pair of exact sequences of vector bundles.

$(1)$ Assume there is a commutative diagram
\begin{equation}
\label{eq:diagram-1}
\vcenter{\xymatrix{
\dots \ar[r] &
E_{i-1} \ar[r] &
E_{i} \ar[r] &
E_{i+1} \ar[r] &
E_{i+2} \ar[r] &
\dots
\\
\dots \ar[r] &
F_{i-1} \ar[r] &
F_{i} \ar[r] &
F_{i+1} \ar[r] \ar[u]_{f_{i+1}} &
F_{i+2} \ar[r] \ar[u]_{f_{i+2}} &
\dots \ar@{}[u]|{\textstyle\dots}
}}
\end{equation}
If $\operatorname{Ext}^\bullet(F_i,E_j) = 0$ for all $j < i$ then there is a unique morphism $f_i \colon F_i \to E_i$
such that the square to the right of it commutes.
In particular, if $f_{i+1} = 0$, then $f_i = 0$ as well.

$(2)$ Assume there is a commutative diagram
\begin{equation}
\label{eq:diagram-2}
\vcenter{\xymatrix{
\dots \ar[r] &
E_{i-2} \ar[r] &
E_{i-1} \ar[r] &
E_{i} \ar[r] &
E_{i+1} \ar[r] &
\dots
\\
\dots \ar[r] \ar@{}[u]|{\textstyle\dots} &
F_{i-2} \ar[r] \ar[u]_{f_{i-2}}  &
F_{i-1} \ar[r] \ar[u]_{f_{i-1}}  &
F_{i} \ar[r] &
F_{i+1} \ar[r] &
\dots
}}
\end{equation}
If $\operatorname{Ext}^\bullet(F_j,E_i) = 0$ for all $j > i$ then there is a unique morphism $f_i \colon F_i \to E_i$
such that the square to the left of it commutes.
In particular, if $f_{i-1} = 0$, then $f_i = 0$ as well.
\end{lemma}
\begin{proof}
Let us prove (1); the proof of (2) is similar.
Denote by $E'$ the kernel of $E_{i+1} \to E_{i+2}$, or equivalently, the cokernel of~$E_{i-1} \to E_i$.
By commutativity of~\eqref{eq:diagram-1}, the composition $F_i \to F_{i+1} \xrightarrow{\ f_{i+1}\ } E_{i+1}$
factors through a morphism $f' \colon F_i \to E'$.
On the other hand, applying the functor $\operatorname{Hom}(F_i,-)$ to the exact sequence
\begin{equation*}
\dots \to E_{i-1} \to E_i \to E' \to 0
\end{equation*}
and using the semiorthogonality $\operatorname{Ext}^\bullet(F_i,E_j) = 0$ for $j < i$, we conclude that the morphism
$f'$ lifts in a unique way to a morphism $f_i \colon F_i \to E_i$.
The square in the diagram~\eqref{eq:diagram-1} formed by the morphisms~$f_i$ and~$f_{i+1}$ commutes by construction,
and also by construction the morphism $f_i$ with this property is unique.
\end{proof}

Now we apply the lemma to construct the required morphism of complexes.
In a contrast with the morphisms discussed above,
this morphism of complexes is \emph{not} restricted from the Grassmannian~$\Gr(3,V)$.

\begin{proposition}
\label{bicomplex}
Consider the exact sequences~\eqref{29} and~\eqref{34} with the grading in which the degree zero terms
are $\Lambda^2V \otimes \cO$ and $\Lambda^4V \otimes \cO$ respectively.
Then there is a unique up to rescaling nonzero $\operatorname{Sp}(V)$-equivariant morphism of complexes from~\eqref{29} to~\eqref{34}.
The components
\begin{equation*}
S^2\cU^\vee \to V \otimes \Sigma^{2,1}\cU^\vee
\qquad\text{and}\qquad
\Lambda^7V \otimes \Sigma^{2,2}\cU^\vee(-3) \to \Sigma^{3,2}\cU^\vee(-3)
\end{equation*}
of this morphism of complexes are both nonzero.
\end{proposition}

\begin{proof}
Consider the following $\operatorname{Sp}(V)$-equivariant composition of morphisms
\begin{equation}
\label{eq:composition}
S^2\cU^{\vee} \to
\cU\otimes \Sigma^{2,1}\cU^{\vee} \to
V\otimes \Sigma^{2,1}\cU^{\vee} \to
\cU^{\vee}\otimes \Sigma^{2,1}\cU^{\vee} \to
\Sigma^{3,1}\cU^\vee,
\end{equation}
where the first arrow is the embedding of a direct summand,
the last is the projection onto a direct summand,
and the middle part is obtained by tensoring the right side of~\eqref{eq:upu-monad} with $\Sigma^{2,1}\cU^\vee$.
Since the right side of~\eqref{eq:upu-monad} is a complex, the composition~\eqref{eq:composition} is zero,
while the composition of the first two arrows in it is injective.
On the other hand, using Lemma~\ref{19} we compute
\begin{equation*}
\operatorname{Hom}(S^2\cU^\vee, V \otimes \Sigma^{2,1}\cU^\vee) \cong
V \otimes \operatorname{Hom}(S^2\cU^\vee, \Sigma^{2,1}\cU^\vee) \cong
V \otimes V,
\end{equation*}
hence the morphism $S^2\cU^\vee \to V \otimes \Sigma^{2,1}\cU^\vee$ defined by the first two arrows is the unique
up to rescaling nonzero $\operatorname{Sp}(V)$-equivariant morphism.

Similarly, the composition $V \otimes \Sigma^{2,1}\cU^\vee \to \Sigma^{3,1}\cU^\vee$
of the last two arrows in~\eqref{eq:composition} is surjective, and it is the unique nonzero $\operatorname{Sp}(V)$-equivariant morphism,
hence coincides with the last morphism in~\eqref{34}.
Thus, we obtain a commutative square
\begin{equation*}
\xymatrix{
V \otimes \Sigma^{2,1}\cU^\vee \ar[r] &
\Sigma^{3,1}\cU^\vee \\
S^2\cU^\vee \ar[r] \ar[u] &
0. \ar[u]
}
\end{equation*}
Applying iteratively Lemma~\ref{lemma:extension}(1)
and using semiorthogonalities of Lemma~\ref{19} and Lemma~\ref{lemma:32},
we extend this square to a morphism of complexes from~\eqref{29} to~\eqref{34}.
The extension is unique by Lemma~\ref{lemma:extension}(1), hence is $\operatorname{Sp}(V)$-equivariant.

Assume the component $\Lambda^7V \otimes \Sigma^{2,2}\cU^\vee(-3) \to \Sigma^{3,2}\cU^\vee(-3)$ of the constructed morphism of complexes is zero.
Then applying Lemma~\ref{lemma:extension}(2) several times
and using again semiorthogonalities of Lemma~\ref{19},
we conclude that all other components are zero as well,
hence so is the component $S^2\cU^\vee \to V \otimes \Sigma^{2,1}\cU^\vee$, which contradicts the construction of the morphism.

Finally, assume there is another equivariant morphism of complexes from~\eqref{29} to~\eqref{34}.
If its component~$S^2\cU^\vee \to V \otimes \Sigma^{2,1}\cU^\vee$ is nonzero,
then the argument above shows that after rescaling it is given by the composition of the first two arrows in~\eqref{eq:composition},
hence by Lemma~\ref{lemma:extension}(1) the morphism coincides with the one that we constructed above.
If, however, the component is zero, the corresponding morphism of complexes is zero as well, again by Lemma~\ref{lemma:extension}(1).
\end{proof}

From now on we fix a morphism of complexes from~\eqref{29} to~\eqref{34} constructed by Proposition~\ref{bicomplex},
and consider it as a bicomplex with two rows.
We will use some of its truncations: the first is
(for convenience in what follows we identify $\Lambda^{k}V$ with~$\Lambda^{8-k}V$ via $\omega$)
\begin{equation}
\label{eq:first-right}
\vcenter{\xymatrix{
*+[F]{\Lambda^4V\otimes\mathcal{O}} \ar[r] &
\Lambda^2V\otimes \Lambda^2\cU^{\vee} \ar[r] &
V\otimes \Sigma^{2,1}\cU^{\vee} \ar[r] &
\Sigma^{3,1}\cU^{\vee} \ar[r] &
0
\\
\Lambda^2V\otimes\mathcal{O} \ar[u] \ar[r] &
V\otimes \cU^{\vee} \ar[u] \ar[r] &
S^2\cU^{\vee} \ar[u] \ar[r] &
0,
}}
\end{equation}
which is quasiisomorphic to
\begin{equation}
\label{eq:first-left}
\vcenter{\xymatrix{
&
0 \ar[r] &
\Sigma^{3,2}\cU^{\vee}(-3) \ar[r] &
V\otimes \Sigma^{2,1}\cU^{\vee}(-2) \ar[r] &
*+[F]{\Lambda^2 V\otimes\cU^{\vee}(-1)} \\
0 \ar[r] &
\Sigma^{3,3}\cU^{\vee}(-4) \ar[r] \ar[u] &
V\otimes\Sigma^{2,2}\cU^{\vee}(-3) \ar[u] \ar[r] &
\Lambda^2 V\otimes \Lambda^2\cU^{\vee}(-2) \ar[u] \ar[r] &
\Lambda^3 V\otimes \cO(-1), \ar[u]
}}
\end{equation}
where the boxed terms are considered to be sitting in degree zero.
The second truncation is
\begin{equation}
\label{eq:second-right}
\vcenter{\xymatrix{
*+[F]{\Lambda^2 V\otimes\cU^{\vee}(-1)} \ar[r] &
\Lambda^4V\otimes\mathcal{O} \ar[r] &
\Lambda^2V\otimes \Lambda^2\cU^{\vee} \ar[r] &
V\otimes \Sigma^{2,1}\cU^{\vee} \ar[r] &
\Sigma^{3,1}\cU^{\vee} \ar[r] &
0
\\
\Lambda^3 V\otimes \cO(-1) \ar[u] \ar[r] &
\Lambda^2V\otimes\mathcal{O} \ar[u] \ar[r] &
V\otimes \cU^{\vee}\ar[u] \ar[r] &
S^2\cU^{\vee} \ar[u] \ar[r] &
0 \ar[u]
}}
\end{equation}
which is quasiisomorphic to
\begin{equation}
\label{eq:second-left}
\vcenter{\xymatrix{
&
0 \ar[r] &
\Sigma^{3,2}\cU^{\vee}(-3) \ar[r] &
*+[F]{V \otimes \Sigma^{2,1}\cU^{\vee}(-2)}
\\
0 \ar[r] &
\Sigma^{3,3}\cU^{\vee}(-4) \ar[r] \ar[u]&
V\otimes\Sigma^{2,2}\cU^{\vee}(-3) \ar[u] \ar[r] &
\Lambda^2 V\otimes \Lambda^2\cU^{\vee}(-2) \ar[u]
}}
\end{equation}
with the same convention about the grading.

\subsection{Some properties of the bicomplex}

Here we describe some of the vertical maps in the bicomplex.

\begin{lemma}
\label{lemma:lambda-injective}
The morphism $\Lambda^2V \otimes \cO \to \Lambda^4V \otimes \cO$
in~\eqref{eq:first-right} is injective.
\end{lemma}
\begin{proof}
Consider the composition
\begin{equation}
\label{eq:composition-2}
V \otimes \cU^\vee \to
V \otimes \cU \otimes \Lambda^2\cU^\vee \to
\Lambda^2V \otimes \Lambda^2\cU^\vee,
\end{equation}
where the first arrow is induced by the embedding of a direct summand $\cU^\vee \to \cU \otimes \Lambda^2\cU^\vee$,
and the second arrow is induced by the wedge product map $V \otimes \cU \to \Lambda^2V \otimes \cO$.
It is easy to check that the square formed by this arrow and the arrow $S^2\cU^\vee \to V \otimes \Sigma^{2,1}\cU^\vee$
constructed in the proof of Proposition~\ref{bicomplex} commutes,
hence by the construction of the morphism of complexes in Proposition~\ref{bicomplex},
this composition is its component.
Note that the morphism induced by~\eqref{eq:composition-2} on global sections is injective.
Indeed,
\begin{equation*}
\cU \otimes \Lambda^2\cU^\vee \cong \cU^\vee \oplus S^2\cU(1)
\end{equation*}
and the second summand is acyclic by Lemma~\ref{19}, hence the first arrow in~\eqref{eq:composition-2} induces an isomorphism on global sections.
On the other hand, the kernel of the second arrow in~\eqref{eq:composition-2} is $S^2\cU \otimes \Lambda^2\cU^\vee$,
which is also acyclic by Lemma~\ref{19}, hence the morphism induced by this arrow on global sections is injective.

Since the map $\Lambda^2V \otimes \cO \to V \otimes \cU^\vee$ in the complex~\eqref{29} induces an injection on global sections
(because all the terms to the left of $\Lambda^2V \otimes \cO$ are acyclic by Lemma~\ref{19} and Lemma~\ref{T1ort}),
it follows that the composition
\begin{equation*}
\Lambda^2V \to \Lambda^4V \to \mathrm{H}^0(X, \Lambda^2V \otimes \Lambda^2\cU^\vee)
\end{equation*}
of the map on global sections induced by the left vertical arrow and the upper left horizontal arrow in~\eqref{eq:first-right} is injective too,
hence the first of these maps is injective, and therefore the left vertical arrow in~\eqref{eq:first-right} is injective.
%
%
%
%
%
\end{proof}

\begin{lemma}
\label{lemma:s22-s32}
The vertical arrows
\begin{equation*}
V \otimes \Sigma^{2,2}\cU^\vee(-3) \to \Sigma^{3,2}\cU^\vee(-3)
\qquad\text{and}\qquad
\Lambda^2V \otimes \Lambda^2\cU^\vee(-2) \to V \otimes \Sigma^{2,1}\cU^\vee(-2)
\end{equation*}
in~\eqref{eq:second-left} coincide up to a twist with the morphisms in~\eqref{staircase} and~\eqref{34} respectively.
In particular, the first is surjective, and the cokernel of the second is isomorphic to $\Sigma^{3,1}\cU^\vee(-2)$.
\end{lemma}

\begin{proof}
We have
\begin{equation*}
\operatorname{Hom}(V \otimes \Sigma^{2,2}\cU^\vee(-3), \Sigma^{3,2}\cU^\vee(-3)) \cong
V \otimes \operatorname{Hom}(\Sigma^{2,2}\cU^\vee, \Sigma^{3,2}\cU^\vee) \cong
V \otimes V,
\end{equation*}
so there is a unique nonzero $\operatorname{Sp}(V)$-equivariant map $V \otimes \Sigma^{2,2}\cU^\vee(-3) \to \Sigma^{3,2}\cU^\vee(-3)$,
hence the one in~\eqref{eq:second-left} coincides with the twist of the one in~\eqref{staircase}, hence it is surjective.

Note that this map can be written quite explicitly as the following composition
\begin{equation*}
V \otimes \Sigma^{2,2}\cU^\vee(-3) \to \cU^\vee \otimes \Sigma^{2,2}\cU^\vee(-3) \to \Sigma^{3,2}\cU^\vee(-3),
\end{equation*}
where the first map is induced by the projection $V \otimes \cO \to \cU^\vee$ and the second map is the projection
onto a direct summand, is yet another nonzero $\operatorname{Sp}(V)$-equivariant map, hence coincides with the one above,
in particular, this map is surjective.

Similarly, one can define a composition
\begin{equation*}
\Lambda^2V \otimes \Lambda^2\cU^\vee(-2) \to V \otimes \cU^\vee \otimes \Lambda^2\cU^\vee(-2) \to V \otimes \Sigma^{2,1}\cU^\vee(-2),
\end{equation*}
where the first arrow is induced by the natural map $\Lambda^2V \otimes \cO \to V \otimes \cU^\vee$
and the second arrow is induced by the projection onto a direct summand $\cU^\vee \otimes \Lambda^2\cU^\vee(-2) \to \Sigma^{2,1}\cU^\vee(-2)$.
Furthermore, there is a natural map from the first chain of morphisms to the second,
and it is esy to check that the corresponding diagram commutes.
Therefore,
by construction of the morphism of complexes in Proposition~\ref{bicomplex},
this composition coincides with the map in~\eqref{eq:second-left}.

It remains to note that this composition is the unique $\mathrm{GL}(V)$-equivariant morphism, hence
coincides (up to a twist) with the map in~\eqref{34}.
In particular, its cokernel is isomorphic to $\Sigma^{3,1}\cU^\vee(-2)$.
\end{proof}

\section{The non-rectangular part}
\label{section:non-rectangular}

For an equivariant vector bundle $E$ on $\IGr(3,V)$ denote by $\mathbf{ss}(E)$ the associated semisimple vector bundle,
that is the vector bundle whose associated representation of the parabolic subgroup in $\operatorname{Sp}(V)$
is the direct sum of all semisimple factors of the representation corresponding to $E$.

\subsection{Extra objects}

Denote the object represented by either of the bicomplexes~\eqref{eq:first-right} or~\eqref{eq:first-left} by $T$,
and the object represented by~\eqref{eq:second-right} and~\eqref{eq:second-left} by $T'$.
Below we will show that $T$ is exceptional.
To start with, we describe the cohomology sheaves of $T$ and $T'$, and a triangle relating them.

\begin{lemma} \label{T_1}
The object $T$ is a vector bundle, whose associated semisimple bundle is
\begin{equation*}
\mathbf{ss}(T) \quad\cong\quad
\cO(-1) \otimes \cS \quad\oplus\quad \cU^\vee(-1) \quad\oplus\quad S^2\cU^\vee(-1) \otimes \cS \quad\oplus\quad
\Sigma^{2,1}\cU^\vee(-1).
\end{equation*}
\end{lemma}
\begin{proof}
First of all note that the bottom row of the bicomplex~\eqref{eq:first-right}
coincides with the Koszul complex~\eqref{15} for $k = 2$, hence is quasi-isomorphic to $\Lambda^2\cU^{\perp}$.
Its associated semisimple bundle is
\begin{equation*}
\mathbf{ss}(\Lambda^2\cU^\perp) \cong \Lambda^2\cU \quad\oplus\quad \cU \otimes \cS \quad\oplus\quad \cO.
\end{equation*}

Now consider the bicomplex
\begin{equation*}
\xymatrix{
&
\Lambda^3V\otimes\mathcal{U}^{\vee} \ar[d] \ar[r]& \Lambda^2 V\otimes
\cU^{\vee} \otimes \cU^{\vee}\ar[d] \ar[r]&  V \otimes S^2
\cU^{\vee}\otimes \cU^{\vee}\ar[d] \ar[r]& S^3\cU^{\vee}\otimes
\cU^{\vee}\ar[d] \ar[r]& 0 \\
\Lambda^4V\otimes\mathcal{O} \ar[r] &
\Lambda^3V\otimes \cU^{\vee} \ar[r] &
\Lambda^2V\otimes S^2\cU^{\vee}
\ar[r]& V\otimes S^3\cU^{\vee}  \ar[r] &
S^4\cU^{\vee} \ar[r] &
0,
}
\end{equation*}
where the top line is the Koszul complex~\eqref{15} for $k = 3$ tensored with $\cU^\vee$,
the bottom line is the Koszul complex~\eqref{15} for $k = 4$,
and all the vertical arrows are induced by the multiplication maps~$S^p\cU^\vee \otimes \cU^\vee \to S^{p+1}\cU^\vee$.
From the spectral sequence starting with vertical differentials it is easy to see that this bicomplex
is quasiisomorphic to the top row in~\eqref{eq:first-right}.
Using the spectral sequence starting with horizontal differentials,
we conclude that the cohomology of the top row in~\eqref{eq:first-right} has a filtration
with factors $\Lambda^4\cU^\perp$ and $\Lambda^3\cU^\perp \otimes \cU^\vee$ respectively.
Therefore, the associated semisimple bundle of this cohomology is isomorphic to
\begin{equation*}
\cO(-1) \otimes \cS \quad \oplus \quad
\Lambda^2\cU \quad \oplus \quad
\Lambda^2\cU \quad \oplus \quad
\Lambda^2\cU\otimes\cU^{\vee}\otimes \cS \quad\oplus\quad
\cU\otimes \cU^{\vee}.
\end{equation*}
Note also that
\begin{align*}
\Lambda^2\cU\otimes\cU^{\vee}\otimes \cS &\cong \cU \otimes \cS \quad\oplus\quad S^2\cU^\vee(-1) \otimes \cS
\qquad\text{and}\qquad\\
\cU\otimes \cU^{\vee} &\cong \cO \quad\oplus\quad \Sigma^{2,1}\cU^\vee(-1).
\end{align*}
By Lemma~\ref{lemma:lambda-injective} the morphism from the cohomology of the bottom row of~\eqref{eq:first-right}
to its top row is injective, hence $T$ is a vector bundle and its semisimple factors are given by the formula in the lemma.
\end{proof}

\begin{lemma}
\label{T_2}
The object $T'$ has two cohomology sheaves, with
\begin{equation*}
\mathcal{H}^0(T') \cong \Sigma^{3,1}\cU^\vee(-2)
\end{equation*}
and with the associated semisimple bundle of $\mathcal{H}^{-1}(T')$ isomorphic to
\begin{equation*}
\mathbf{ss}(\mathcal{H}^{-1}(T')) \quad\cong\quad
\cO(-2) \quad\oplus\quad \cU^\vee(-2) \otimes \cS \quad\oplus\quad \cU(-1) \quad\oplus\quad \cO(-1) \otimes \cS.
\end{equation*}
\end{lemma}
\begin{proof}
Using exactness of the rows of \eqref{eq:second-left} we see that cohomology sheaves of $T'$ are concentrated in the degrees~$-1$ and~$0$.
To compute $\mathcal{H}^0(T')$ and $\mathcal{H}^{-1}(T')$, let us use the spectral sequence of~\eqref{eq:second-left}
that starts with vertical differentials.
By Lemma~\ref{lemma:s22-s32}
the first page of the spectral sequence looks like
\begin{equation}
\label{spectral}
\vcenter{\xymatrix@R=1ex{
0  &
0  &
\Sigma^{3,1}\cU^{\vee}(-2)
\\
\Sigma^{3,3}\cU^{\vee}(-4) \ar[r]&
K_1 \ar[r] &
K_2,
}}
\end{equation}
where
\begin{equation*}
K_1 = \mathrm{Ker}(V\otimes\Sigma^{2,2}\cU^{\vee}(-3)\to \Sigma^{3,2}\cU^{\vee}(-3))
\quad\text{and}\quad
K_2 = \mathrm{Ker}(\Lambda^2 V\otimes \Lambda^2\cU^{\vee}(-2)\to V\otimes \Sigma^{2,1}\cU^{\vee}(-2)).
\end{equation*}
Clearly, the upper right term $\Sigma^{3,1}\cU^{\vee}(-2)$ in~\eqref{spectral} survives in the spectral sequence and gives $\mathcal{H}^0(T')$.
Since the only other cohomology of $T'$ sits in degree $-1$, it follows that the bottom row in~\eqref{spectral} is left exact.
It is easy to see that the semisimple bundle associated with $K_1$ is
%
%
%
\begin{equation*}
\Sigma^{3,3}\cU^{\vee}(-4) \quad\oplus\quad
\Sigma^{2,1}\cU^{\vee}(-3) \quad\oplus\quad
\Sigma^{2,2}\cU^{\vee}(-3)\otimes \cS \quad\oplus\quad
\Lambda^2\cU^{\vee}(-2),
\end{equation*}
and the semisimple bundle associated with $K_2$ is
\begin{multline*}
\cO(-2)\quad\oplus\quad
\Sigma^{2,1}\cU^{\vee}(-3)\quad\oplus\quad
\cU^\vee(-2) \otimes \cS \quad\oplus\quad
\\
\quad\oplus\quad
\Sigma^{2,2}\cU^{\vee}(-3)\otimes \cS \quad\oplus\quad
\cU(-1) \quad\oplus\quad \cU(-1) \quad\oplus\quad
\cO(-1) \otimes \cS.
\end{multline*}
By left exactness all common factors are canceled, hence the semisimple bundle associated with $\mathcal{H}^{-1}(T')$
is given by the formula in the lemma.
%
%
\end{proof}

Note that by definition of the objects $T$ and $T'$ we have a distinguished triangle
\begin{equation}
\label{eq:tt}
T[-1] \to T' \to G
\end{equation}
(induced by the embedding of~\eqref{eq:first-right} into~\eqref{eq:second-right}),
where $G$ is the two-term complex
\begin{equation*}
\Lambda^3V \otimes \cO(-1) \to \Lambda^2V \otimes \cU^\vee(-1).
\end{equation*}

\subsection{Exceptionality}

Here we will prove that $T$ is an exceptional bundle.
We will need the following

\begin{lemma}
\label{lemma:t-s21}
The vector bundle $T$ is right orthogonal to $\Sigma^{2,1}\cU^{\vee}(-1).$
\end{lemma}
\begin{proof}
Using the bicomplex \eqref{eq:first-left} and Lemma \ref{19} we see that $\mathrm{Ext}^{\bullet}(\Sigma^{2,1}\cU^{\vee}(-1),T)$
is computed by an application of $\mathrm{Ext}^{\bullet}(\Sigma^{2,1}\cU^{\vee}(-1),-)$ to the complex
\begin{equation*}
\Sigma^{3,3}\cU^{\vee}(-4) \to
V\otimes\Sigma^{2,2}\cU^{\vee}(-3) \to
\Sigma^{3,2}\cU^{\vee}(-3).
\end{equation*}
The proof of Proposition~\ref{bicomplex} shows that this complex coincides
with the tensor product of~\eqref{eq:upu-monad} with~$\Sigma^{2,2}\cU^{\vee}(-3)$
composed with the direct summand embedding $\Sigma^{3,3}\cU^{\vee}(-4) \hookrightarrow \cU \otimes \Sigma^{2,2}\cU^{\vee}(-3)$
and projection $\cU^\vee \otimes \Sigma^{2,2}\cU^{\vee}(-3) \twoheadrightarrow \Sigma^{3,2}\cU^{\vee}(-3)$.
The other direct summands $\Sigma^{2,1}\cU^{\vee}(-3)$ and $\Lambda^2\cU^\vee(-2)$ are right semiorthogonal to $\Sigma^{2,1}\cU^{\vee}(-1)$
by Lemma~\ref{19}, so to deduce that $\mathrm{Ext}^{\bullet}(\Sigma^{2,1}\cU^{\vee}(-1),T) = 0$ it is enough to check that
$\mathrm{Ext}^{\bullet}(\Sigma^{2,1}\cU^{\vee}(-1),\Sigma^{2,2}\cU^{\vee}(-3) \otimes \cS) = 0$.
But this is proved in Lemma~\ref{21(-1)}.
\end{proof}

Recall the collections $\mathfrak{E}$ and $\mathfrak{E}'$ defined in~\eqref{eq:e} and~\eqref{eq:e-prime} respectively.
Also recall that by Corollary~\ref{corollary:e} and~\ref{corollary:e-prime} they are starting blocks
of rectangular Lefschetz collections of length~6.
Below we will repeatedly use this fact.

\begin{lemma} \label{9}
The vector bundle $T$ satisfies the following properties:
\begin{enumerate}
\item $T$ is isomorphic to a shift of the left mutation of $\Sigma^{3,1}\cU^{\vee}$ through $\mathfrak{E}$;
more precisely
\begin{equation*}
T \simeq \mathbb{L}_{\mathfrak{E}}(\Sigma^{3,1}\cU^{\vee})[-3].
\end{equation*}
\item $T$ is exceptional.
\item The collection $(T,\mathfrak{E}',\mathfrak{E}'(1))$ is exceptional.
\end{enumerate}
\end{lemma}
\begin{proof}
Let us prove (i). It is obvious from \eqref{eq:first-right} that there is a morphism $\Sigma^{3,1}\cU^\vee[-3] \to T$
whose cone is in $\mathfrak{E}$.
Hence to prove that $T\simeq \mathbb{L}_{\mathfrak{E}}\Sigma^{3,1}\cU^{\vee}[-3]$ it is enough to check that $T\subset \mathfrak{E}^{\perp}.$

Using the bicomplex~\eqref{eq:first-left}, Lemma~\ref{19} and Lemma \ref{T1ort} we see that $T\subset\langle \mathcal{O},\cU^{\vee},\Lambda^{2}\cU^{\vee},\Sigma^{2,1}\cU^{\vee}\rangle^{\perp}.$
So, it only remains to prove that
$\mathrm{Ext}^{\bullet}(S^2\cU^{\vee},T)=0$.
Indeed, from the bicomplex~\eqref{eq:first-right} and Lemma~\ref{19} we see
that~$\mathrm{Ext}^{\bullet}(S^2\cU^{\vee},T)$ is quasi-isomorphic to the complex
\begin{equation*}
\xymatrix@R=2ex{
\mathrm{Hom}(S^2\cU^{\vee},S^2\cU^{\vee}) \ar[r] \ar@{=}[d]& \mathrm{Hom}(S^2\cU^{\vee},V\otimes\Sigma^{2,1}\cU^{\vee})\ar[r]\ar@{=}[d]& \mathrm{Hom}(S^2\cU^{\vee},\Sigma^{3,1}\cU^{\vee}) \ar@{=}[d]\\
\Bbbk\ar [r]& V\otimes V\ar[r] & (V\otimes V)/\omega.
}
\end{equation*}
For the last isomorphism in the diagram see Lemma \ref{all31}.
The first map is induced by the injective morphism $S^2\cU^{\vee} \to V \otimes \Sigma^{2,1}\cU^{\vee}$
(see the proof of Proposition~\ref{bicomplex}), hence is injective.
On the other hand, all Ext-groups from $S^2\cU^{\vee}$ to the exact complex~\eqref{34} vanish,
so using the corresponding hypercohomology spectral sequence and Lemma~\ref{19} and Lemma~\ref{lemma:32},
we can check that the second map is surjective.
It follows that the complex is exact, hence $\mathrm{Ext}^{\bullet}(S^2\cU^{\vee},T) = 0$.

Let us prove (ii).
By Lemma~\ref{31U} the bundle $\Sigma^{3,1}\cU^{\vee}$ is exceptional.
Since $T\simeq \mathbb{L}_{\mathfrak{E}}\Sigma^{3,1}\cU^{\vee}[-3]$ by (i)
and~$\Sigma^{3,1}\cU^{\vee}\in {}^{\perp}\mathfrak{E}$ by Corollary~\ref{17},
we conclude that $T$ is exceptional as well.

Finally, let us prove (iii).
From Lemma~\ref{18} and the definition of $\mathfrak{E}'$, it follows that $\Sigma^{3,1}\cU^\vee$
is right orthogonal to $\langle \mathfrak{E}(1) \rangle$, hence by Lemma~\ref{19} the same is true for $T$.
Moreover, by part (i) we know that~$T$ is right orthogonal to~$\mathfrak{E}$.
Since obviously $\langle \fE', \fE'(1) \rangle \subset \langle \Sigma^{2,1}\cU^\vee(-1), \fE, \fE(1) \rangle$,
it only remains to show that~$T$ is right orthogonal to $\Sigma^{2,1}\cU^\vee(-1)$.
But this was proved in Lemma~\ref{lemma:t-s21}.
\end{proof}

\begin{corollary}
\label{corollary:t1}
We have
\begin{equation*}
T \cong \mathbb{L}_{\mathfrak{E}',\mathfrak{E}'(1)}(\Sigma^{3,1}\cU^\vee[-3]).
\end{equation*}
\end{corollary}
\begin{proof}
Follows from parts (i) and (iii) of Lemma~\ref{9} since obviously $\fE \subset \langle \fE', \fE'(1) \rangle$.
\end{proof}

The crucial computation is given by the following

\begin{proposition} \label{T}
We have
\begin{equation*}
\mathbb{L}_{\mathfrak{E}'(-2),\mathfrak{E}'(-1)}(T) \cong T(-2)[4].
\end{equation*}
\end{proposition}
\begin{proof}
Recall that by Lemma~\ref{T_2} we have a natural morphism
\begin{equation}
\label{eq:t2-morphism}
T' \to \mathcal{H}^0(T') \cong \Sigma^{3,1}\cU^\vee(-2).
\end{equation}
Note that its cone, $\mathcal{H}^{-1}(T')[2]$, is contained in the subcategory $\langle \mathfrak{E}'(-2), \mathfrak{E}'(-1) \rangle$.
Indeed, this follows from the description of its associated semisimple bundle in Lemma~\ref{T_2}
together with the evident inclusions
\begin{equation*}
\cO(-2),\ \cU(-1) \in \langle \mathfrak{E}'(-2) \rangle
\end{equation*}
(since $\cU(-1) \cong \Lambda^2\cU^\vee(-2)$),
and with slightly less evident inclusions (using~\eqref{eq:upu-monad})
\begin{align*}
\cO(-1) \otimes \cS &\cong
\Big\{ \cU(-1) \to V \otimes \cO(-1) \to \cU^\vee(-1) \Big\} &&
\in \langle \mathfrak{E}'(-2), \mathfrak{E}'(-1) \rangle
\\
\intertext{and}
\cU^\vee(-2) \otimes \cS &\cong
\Big\{ \cU \otimes \cU^\vee(-2) \to V \otimes \cU^\vee(-2) \to \cU^\vee \otimes \cU^\vee(-2) \Big\} &&
\in \langle \mathfrak{E}'(-2) \rangle
\end{align*}
(since $\cU \otimes \cU^\vee(-2) \cong \Sigma^{2,1}\cU^\vee(-3) \oplus \cO(-2)$
and $\cU^\vee \otimes \cU^\vee(-2) \cong S^2\cU^\vee(-2) \oplus \Lambda^2\cU^\vee(-2)$).
Therefore, applying the mutation functor $\mathbb{L}_{\mathfrak{E}'(-2),\mathfrak{E}'(-1)}$ to the morphism~\eqref{eq:t2-morphism},
we conclude that
\begin{equation*}
\mathbb{L}_{\mathfrak{E}'(-2),\mathfrak{E}'(-1)}(T') \cong
\mathbb{L}_{\mathfrak{E}'(-2),\mathfrak{E}'(-1)}(\Sigma^{3,1}\cU^\vee(-2)) \cong T(-2)[3],
\end{equation*}
where the second isomorphism follows from Corollary~\ref{corollary:t1}.

On the other hand, applying the functor $\mathbb{L}_{\mathfrak{E}'(-2),\mathfrak{E}'(-1)}$ to the triangle~\eqref{eq:tt}
and taking into account that we have $G \in \langle \mathfrak{E}'(-1) \rangle$, we conclude that
\begin{equation*}
\mathbb{L}_{\mathfrak{E}'(-2),\mathfrak{E}'(-1)}(T') \cong \mathbb{L}_{\mathfrak{E}'(-2),\mathfrak{E}'(-1)}(T[-1]).
\end{equation*}
From these two observations, we deduce
\begin{equation*}
\mathbb{L}_{\mathfrak{E}'(-2),\mathfrak{E}'(-1)}(T[-1]) \cong T(-2)[3]
\end{equation*}
which finally proves the lemma.
%
%
%
%
\end{proof}

\subsection{Exceptional collection}

Now we are ready to construct the first version of the exceptional collection from Theorem~\ref{intro:main}.

\begin{proposition}
\label{proposition:ec-t}
The following two collections of vector bundles
\begin{align}
\label{preanother-collection}
&\langle T, \fE', T(1), \fE'(1), \fE'(2), \fE'(3), \fE'(4), \fE'(5) \rangle,\qquad\text{and}\\
\label{another-collection}
&\langle T, \fE', \fE'(1), \fE'(2), T(3), \fE'(3), \fE'(4), \fE'(5) \rangle
\end{align}
are exceptional.
They generate the same subcategory of $\Db(\IGr(3,V))$.
%
\end{proposition}

Note that~\eqref{preanother-collection} is Lefshetz with respect to $\cO(1)$, but not rectangular.
On the other hand, \eqref{another-collection} is rectangular with respect to $\cO(3)$.

\begin{proof}
By Proposition~\ref{T} the first collection is obtained from the second by the left mutation of $T(3)$ through $\fE'(2)$ and $\fE'(1)$,
so it is enough to check that~\eqref{another-collection} is exceptional.
Furthermore, by Corollary~\ref{corollary:e-prime} and Lemma~\ref{9}(ii) it is enough to check that $T$ and $T(3)$ are semiorthogonal to $\fE'(i)$.

First, by Lemma~\ref{9}(iii) we know that $(T,\mathfrak{E}', \mathfrak{E}'(1))$ is an exceptional collection.
Furthermore, $\Sigma^{3,1}\cU^\vee$ is right orthogonal to $(\mathfrak{E}'(2), \mathfrak{E}'(3), \mathfrak{E}'(4), \mathfrak{E}'(5))$
by Lemma~\ref{18}, hence by Corollary~\ref{corollary:e-prime} and Corollary~\ref{corollary:t1} so is $T$.
Thus, the collection
\begin{equation*}
T,\mathfrak{E}', \mathfrak{E}'(1), \mathfrak{E}'(2), \mathfrak{E}'(3), \mathfrak{E}'(4), \mathfrak{E}'(5)
\end{equation*}
is exceptional.

Moreover, this argument also proves that $T(3)$ is right orthogonal to $(\mathfrak{E}'(3), \mathfrak{E}'(4), \mathfrak{E}'(5))$,
and by Serre duality it also follows that it is left orthogonal to $(\mathfrak{E}', \mathfrak{E}'(1), \mathfrak{E}'(2))$.
So, it remains to show that $T(3)$ and~$T$ are semiorthogonal.
But this evidently follows from the semiorthogonality of~$\Sigma^{3,1}\cU^\vee(3)$ and~$\Sigma^{3,1}\cU^\vee$, Lemma~\ref{31(3)},
and a combination of Corollary~\ref{corollary:e-prime} and Corollary~\ref{corollary:t1}.
\end{proof}

Now we can pass from~\eqref{preanother-collection} and~\eqref{another-collection} to the other collections of Theorem~\ref{intro:main}.
Recall from Lemma~\ref{T_1} that the last semisimple factor of $T$ is the bundle $\Sigma^{2,1}\cU^\vee(-1)$,
hence we have a canonical epimorphism~$T \to \Sigma^{2,1}\cU^\vee(-1)$
(this morphism can be also constructed from the bicomplex~\eqref{eq:first-left}
and a natural morphism~$\Lambda^2V \otimes \cU^\vee(-1) \to \Sigma^{2,1}\cU^\vee(-1)$).
We denote its kernel by $F$, so that we have an exact sequence
\begin{equation}
\label{eq:def-f}
0 \to F \to T \to \Sigma^{2,1}\cU^\vee(-1) \to 0.
\end{equation}
We prove the following

\begin{lemma}
\label{lemma:tf-mutation}
The sequence~\eqref{eq:def-f} is a right mutation sequence, i.e.,
\begin{equation*}
F \cong \mathbb{R}_{\Sigma^{2,1}\cU^\vee(-1)}(T).
\end{equation*}
In particular, $F$ is an exceptional vector bundle.
Conversely,
\begin{equation*}
T\cong \mathbb{L}_{\Sigma^{2,1}\cU^\vee(-1)}(F).
\end{equation*}
\end{lemma}
\begin{proof}
By Lemma~\ref{18} and Lemma~\ref{9}(i) we have
\begin{equation*}
\operatorname{Ext}^\bullet(\Sigma^{2,1}\cU^\vee(5),T) \cong
\operatorname{Ext}^\bullet(\Sigma^{2,1}\cU^\vee(5),\Sigma^{3,1}\cU^\vee[-3]) \cong \Bbbk[-12].
\end{equation*}
Therefore, by Serre duality and~\eqref{eq:kx} we have
\begin{equation*}
\operatorname{Ext}^\bullet(T,\Sigma^{2,1}\cU^\vee(-1)) \cong \Bbbk.
\end{equation*}
This proves that the right mutation of $T$ through $\Sigma^{2,1}\cU^\vee(-1)$ is the shifted cone
of the unique nontrivial morphism from $T$ to $\Sigma^{2,1}\cU^\vee(-1)$.
Comparing with the definition of $F$, we conclude that this is equal to~$F$.
Since $T$ is right orthogonal to $\Sigma^{2,1}\cU^\vee(-1)$ by Lemma~\ref{lemma:t-s21},
it follows that $F$ is exceptional and that we also have $T\cong \mathbb{L}_{\Sigma^{2,1}\cU^\vee(-1)}(F).$
\end{proof}

\begin{remark}
\label{remark:f}
Using Lemma~\ref{T_1} it is easy to see that $F(1)$ is an iterated extension of the bundles
\begin{equation*}
\cO \otimes \cS,\qquad
\cU^\vee,\quad\text{and}\quad
S^2\cU^\vee \otimes \cS.
\end{equation*}
It can be deduced from this that $F(1)$ coincides with the exceptional bundle $\mathcal{E}^{2,0,0;1}$ from~\cite{1}.
However, we will not use this fact, so we leave it without a proof.
\end{remark}

\begin{corollary} \label{FF}
We have $\mathbb{L}_{\mathfrak{E}(-2),\mathfrak{E}(-1)}(F)\cong F(-2)[4]$.
Conversely, $\mathbb{R}_{\mathfrak{E},\mathfrak{E}(1)}(F)\cong F(2)[-4]$.
\end{corollary}
\begin{proof}
Using Lemma \ref{lemma:tf-mutation} and Proposition \ref{T} we have
\begin{multline*}
\mathbb{L}_{\mathfrak{E}(-2),\mathfrak{E}(-1)}(F)\simeq\mathbb{R}_{\Sigma^{2,1}\cU^{\vee}(-3)}\mathbb{L}_{\mathfrak{E}'(-2),\mathfrak{E}'(-1)}
\mathbb{L}_{\Sigma^{2,1}\cU^{\vee}(-1)}(F)\simeq\\
\simeq\mathbb{R}_{\Sigma^{2,1}\cU^{\vee}(-3)}\mathbb{L}_{\mathfrak{E}'(-2),\mathfrak{E}'(-1)}(T)\simeq
\mathbb{R}_{\Sigma^{2,1}\cU^{\vee}(-3)}(T(-2)[4])\simeq F(-2)[4].
\end{multline*}
The second statement is analogous.
\end{proof}

%
%
%

The main result of this section is the following.
\begin{proposition}
\label{proposition:ec}
The following two collections of vector bundles
\begin{align}
\label{pre1}
&\langle F, \fE, F(1), \fE(1), \fE(2), \fE(3), \fE(4), \fE(5) \rangle,\qquad\text{and}\\
\label{1}
&\langle F, \fE, \fE(1), \fE(2), F(3), \fE(3), \fE(4), \fE(5) \rangle
\end{align}
are exceptional.
They generate the same subcategory of $\Db(\IGr(3,V))$ as the collections~\eqref{preanother-collection} and~\eqref{another-collection}.
%
\end{proposition}

As before, \eqref{pre1} is Lefshetz with respect to $\cO(1)$, but not rectangular,
while~\eqref{1} is rectangular with respect to $\cO(3)$.

\begin{proof}
We prove exceptionality of~\eqref{1}, the case of~\eqref{pre1} is analogous.
Consider~\eqref{another-collection} and
mutate $T$ to the right of $\Sigma^{2,1}\cU^\vee(-1) \in \fE'$ and~$T(3)$ to the right of $\Sigma^{2,1}\cU^\vee(2) \in \fE'(3)$.
By Lemma~\ref{lemma:tf-mutation} we get the objects $F$ and $F(3)$.
Finally, using Proposition~\ref{proposition:long-mutations} replace the starting object $\Sigma^{2,1}\cU^\vee(-1)$ by $\Sigma^{2,1}\cU^\vee(5)$,
to get the required exceptional collection.
\end{proof}

\section{Fullness}
\label{section:fullness}

In this section we prove that the exceptional collections constructed in Propositions~\ref{proposition:ec-t} and~\ref{proposition:ec} are full.
Recall that all these collections generate the same subcategory of $\Db(\IGr(3,V))$, which we denote by~$\cD$.
Thus
\begin{equation}
\label{eq:cd}
\cD = \langle F, \fE, \fE(1), \fE(2), F(3), \fE(3), \fE(4), \fE(5) \rangle \subset \Db(\IGr(3,V))
\end{equation}
and we aim to prove that $\cD = \Db(\IGr(3,V))$.

\subsection{Adding some objects.}

We begin with some preparations.

\begin{lemma}
\label{3}
We have $\Lambda^2(V/\cU) \in \langle \Sigma^{2,1}\cU^{\vee}(-1), \mathfrak{E} \rangle$.
\end{lemma}

\begin{proof}
First, from the complex~\eqref{15} for $k = 2$ it follows that
\begin{equation*}
\Lambda^2\cU^\perp \in \langle \mathfrak{E} \rangle.
\end{equation*}
Further, consider the commutative square
\begin{equation} \label{33}
\vcenter{\xymatrix{
\Lambda^2\cU^{\perp}\ar@{^{(}->}[r]\ar@{->>}[d]& \Lambda^2V\ar@{->>}[d]  \\
\Lambda^2\cS\ar@{^{(}->}[r]& \Lambda^2(V/\cU)
}}
\end{equation}
It gives a complex of the form
\begin{equation}
\label{eq:complex}
\Lambda^2\cU^{\perp} \to \Lambda^2V \oplus \Lambda^2\cS \to \Lambda^2(V/\cU),
\end{equation}
and it is clear that its only cohomology sheaf sits in the middle term.
From the spectral sequence of~\eqref{33} starting with horizontal arrows we conclude
that this cohomology sheaf is isomorphic to
\begin{equation*}
\operatorname{Ker}(\cU^\vee \otimes \cU^\perp \oplus \Lambda^2\cU^\vee \twoheadrightarrow \cU^\vee \otimes \cS \oplus  \Lambda^2\cU^\vee)
\cong \cU^\vee \otimes \cU \cong \Sigma^{2,1}\cU^\vee(-1) \oplus \cO,
\end{equation*}
hence it is contained in $\langle \Sigma^{2,1}\cU^{\vee}(-1), \mathfrak{E} \rangle$.
Finally, note that $\Lambda^2\cS \cong \cO$.
Combining all this we deduce that~$\Lambda^2(V/\cU) \in \langle \Sigma^{2,1}\cU^{\vee}(-1), \mathfrak{E} \rangle$.
\end{proof}

\begin{proposition} \label{2}
We have $\Sigma^{2,2}\cU^{\vee}\in \langle \mathfrak{E}(1), \mathfrak{E}(2) \rangle$.
In particular, $\Sigma^{2,2}\cU^{\vee}(k) \in \cD$ for $-1 \le k \le 3$.
\end{proposition}
\begin{proof}
Note that $\Sigma^{2,2}\cU^{\vee}=S^{2}\cU(2)$.
The dual of~\eqref{15} twisted by $\cO(2)$ gives an exact sequence
\begin{equation*}
0\to S^2\cU(2)\to V\otimes\Lambda^2\cU^{\vee}(1)\to \Lambda^2V\otimes \mathcal{O}(2)\to \Lambda^2(V/\cU)(2)\to 0.
\end{equation*}
Using Lemma~\ref{3} we deduce $\Sigma^{2,2}\cU^{\vee}\in \langle \mathfrak{E}(1), \mathfrak{E}(2) \rangle$.
The second claim follows from the definition of $\cD$.
\end{proof}

\begin{proposition} \label{10}
We have $\Sigma^{3,3}\cU^{\vee}\subset \langle \mathfrak{E}(1),\mathfrak{E}(2),\mathfrak{E}(3),\mathfrak{E}(4) \rangle$.
In particular, $\Sigma^{3,3}\cU^{\vee}(k) \in \cD$ for~$-1 \le k \le 1$.
\end{proposition}
\begin{proof}
The first claim follows from the staircase complex~\eqref{29} and Proposition~\ref{2} and
the second claim follows from the definition of $\cD$.
\end{proof}

Note that the vector bundles from the previous propositions lie in the rectangular part of \eqref{1}.

\begin{proposition} \label{prop:F}
We have $F(k) \in \cD$ for $0 \le k \le 5$.
\end{proposition}
\begin{proof}
The vector bundles $F$ and $F(3)$ lie in $\cD$ by~\eqref{eq:cd}.
By Corollary~\ref{FF} we have
\begin{equation*}
F(2)\subset\langle F,\mathfrak{E},\mathfrak{E}(1)\rangle \subset \cD,\quad
F(1)\subset\langle \mathfrak{E}(1),\mathfrak{E}(2), F(3) \rangle \subset \cD,
\quad\text{and}\quad
F(5)\subset\langle F(3),\mathfrak{E}(3),\mathfrak{E}(4)\rangle \subset \cD.
\end{equation*}
Furthermore, using Corollary~\ref{FF} again we deduce that
\begin{equation*}
F(4)\subset\langle F(2),\mathfrak{E}(2),\mathfrak{E}(3)\rangle \subset
\langle F,\mathfrak{E},\mathfrak{E}(1), \mathfrak{E}(2),\mathfrak{E}(3)\rangle \subset \cD.
\end{equation*}
This finishes the proof.
\end{proof}

\begin{corollary} \label{4}
We have $\Sigma^{3,2}\cU^{\vee}(k) \in \cD$ for $-1 \le k \le 2$ and
$\Sigma^{3,1}\cU^{\vee}(k) \in \cD$ for $1 \le k \le 5$.
\end{corollary}
\begin{proof}
By Lemma~\ref{9}(i) and~\eqref{eq:def-f} we have
$\Sigma^{3,1}\cU^{\vee}(k)\subset \langle \Sigma^{2,1}\cU^{\vee}(k-1),F(k),\mathfrak{E}(k) \rangle$.
So, using Proposition~\ref{prop:F}, we immediately deduce that $\Sigma^{3,1}\cU^{\vee}(k) \in \cD$ for $1 \le k \le 5$.


Now from~\eqref{34} we see that $\Sigma^{3,2}\cU^{\vee}(k) \in \cD$ for $-1 \le k \le 2$.
\end{proof}

\subsection{Fullness}
%

Consider the following diagram (with obvious projections):
\begin{equation*}
\xymatrix{
&\mathrm{IFl}(2,3;V)\ar[dl]^p \ar[dr]_q\\
\mathrm{IGr}(2,V) && \mathrm{IGr}(3,V),
}
\end{equation*}
where $\mathrm{IFl}(2,3,V)$ is the isotropic flag variety.
The projections $p$ and $q$ turn $\mathrm{IFl}(2,3,V)$ into a $\mathbb{P}^3$-bundle over $\mathrm{IGr}(2,V)$
and a $\mathbb{P}^2$-bundle over $\mathrm{IGr}(3,V)$ respectively.
Let us denote by $\cU_3$ and $\cU_2$ the tautological bundles on $\mathrm{IGr}(3,V)$ and $\mathrm{IGr}(2,V)$ respectively.
Then
\begin{equation*}
\cO(H_3)=q^*\Lambda^3\cU_3^{\vee}
\end{equation*}
is the Grothendieck invertible sheaf for the $\mathbb{P}^3$-bundle over $\mathrm{IGr}(2,V)$, and
\begin{equation*}
\cO(H_2)=p^*\Lambda^2\cU_2^{\vee}
\end{equation*}
is the Grothendieck invertible sheaf for the $\mathbb{P}^2$-bundle over $\mathrm{IGr}(3,V)$.

For each $i \in \mathbb{Z}$ denote
\begin{equation*}
\mathrm{D}_i := p^*(\mathrm{D}^{b}(\mathrm{IGr}(2,V)))\otimes \mathcal{O}(iH_3) \subset \mathrm{D}^{b}(\mathrm{IFl}(2,3,V)).
\end{equation*}
By Theorem~\ref{Orlov} applied to the projection~$p$,
there is a semiorthogonal decomposition
\begin{equation*}
\mathrm{D}^{b}(\mathrm{IFl}(2,3,V))=\langle \mathrm{D}_3, \mathrm{D}_4, \mathrm{D}_5, \mathrm{D}_6  \rangle.
\end{equation*}
Next, we choose a convenient full exceptional collection in each of the components above.
Recall from Example~\eqref{14}(iii) that $\mathrm{D}^{b}(\mathrm{IGr}(2,V))$ has a full Lefschetz exceptional collection of length $24$
with the first block equal to $\cO,\cU^{\vee},S^2\cU^{\vee},S^3\cU^{\vee}$ and the support partition $(4,4,4,3,3,3,3)$.
We twist this exceptional collection by $\cO(-5)$ and denote by $\fG$ the resulting full exceptional collection
\begin{equation} \label{21}
\fG=
\left(\begin{array}{rrrrrrr}
S^3\mathcal{U}_2^{\vee}(-5) & S^3\mathcal{U}_2^{\vee}(-4) & S^3\mathcal{U}_2^{\vee}(-3) &&& \\
S^2\mathcal{U}_2^{\vee}(-5)& S^2\mathcal{U}_2^{\vee}(-4) & S^2\mathcal{U}_2^{\vee}(-3)&
S^2\mathcal{U}_2^{\vee}(-2)& S^2\mathcal{U}_2^{\vee}(-1)& S^2\mathcal{U}_2^{\vee} & S^2\mathcal{U}_2^{\vee}(1) \\
\mathcal{U}_2^{\vee}(-5)&\mathcal{U}_2^{\vee}(-4) & \mathcal{U}_2^{\vee}(-3) &
\mathcal{U}_2^{\vee}(-2)& \mathcal{U}_2^{\vee}(-1)& \mathcal{U}_2^{\vee} & \mathcal{U}_2^{\vee}(1) \\
\mathcal{O}(-5)& \mathcal{O}(-4)& \mathcal{O}(-3) &
\mathcal{O}(-2) & \mathcal{O}(-1) & \mathcal{O} & \mathcal{O}(1)
\end{array}\right).
\end{equation}
Twisting this collection by $\cO(-1)$ and applying Proposition~\ref{proposition:long-mutations} to its first object,
we see that
\begin{equation} \label{21prime}
\fG' =
\left(\begin{array}{rrrrrrrr}
S^3\mathcal{U}_2^{\vee}(-6) & S^3\mathcal{U}_2^{\vee}(-5) & S^3\mathcal{U}_2^{\vee}(-4) &&& \\
S^2\mathcal{U}_2^{\vee}(-6)& S^2\mathcal{U}_2^{\vee}(-5) & S^2\mathcal{U}_2^{\vee}(-4)&
S^2\mathcal{U}_2^{\vee}(-3)& S^2\mathcal{U}_2^{\vee}(-2)& S^2\mathcal{U}_2^{\vee}(-1) & S^2\mathcal{U}_2^{\vee} \\
\mathcal{U}_2^{\vee}(-6)&\mathcal{U}_2^{\vee}(-5) & \mathcal{U}_2^{\vee}(-4) &
\mathcal{U}_2^{\vee}(-3)& \mathcal{U}_2^{\vee}(-2)& \mathcal{U}_2^{\vee}(-1) & \mathcal{U}_2^{\vee} \\
& \mathcal{O}(-5)& \mathcal{O}(-4) &
\mathcal{O}(-3) & \mathcal{O}(-2) & \mathcal{O}(-1) & \mathcal{O} & \ \mathcal{O}(1)
\end{array}\right).
\end{equation}
is also a full exceptional collection in $\mathrm{D}^{b}(\mathrm{IGr}(2,V))$.


Now, we consider the following exceptional collection:
\begin{multline} \label{22}
\mathrm{D}^{b}(\mathrm{IFl}(2,3,V))=
\langle \mathrm{D}_3, \mathrm{D}_4, \mathrm{D}_5, \mathrm{D}_6  \rangle
= \\ =
\Big\langle
p^*\fG\otimes \cO(3H_3),
p^*\mathfrak{G}\otimes \cO(4H_3),
p^*\fG\otimes \cO(5H_3),
p^*\fG' \otimes \cO(6H_3)
\Big\rangle.
\end{multline}
%
The collection~\eqref{22} consists of the bundles of the form
\begin{equation*}
p^*S^j \mathcal{U}_2^{\vee}(kH_2)\otimes \mathcal{O}(iH_3),
\end{equation*}
where $i\in [3,6],$ $j\in [0,3]$, and $k\in [-6,1]$.
Below we compute the direct images under $q$ of these sheaves and check that most of them are contained in $\cD$.

\begin{lemma}
\label{direct image}
We have
\begin{equation*}
q_*(p^*S^j \mathcal{U}_2^{\vee}(kH_2)\otimes \mathcal{O}(iH_3))=
\left\{\begin{array}{rl}
\Sigma^{j+k,k}\mathcal{U}_3^{\vee}(iH_3), & \text{if}\ k+2>1;\\
\Sigma^{j-1,-k-2}\mathcal{U}_3^{\vee}((i+k+1)H_3)[-1], & \text{if}\ j+k+3>1>k+2;\\
\Sigma^{-k-3,j}\mathcal{U}_3^{\vee}((i+k+1)H_3)[-2], & \text{if}\ 1>j+k+3.
\end{array}\right.
\end{equation*}
\end{lemma}
\begin{proof}
By the projection formula we have
$q_*(p^*S^j \mathcal{U}_2^{\vee}(-kH_2)\otimes \mathcal{O}(iH_3))=
q_*p^*S^j \mathcal{U}_2^{\vee}(-kH_2) \otimes \mathcal{O}(iH_3)$.
So it is enough to compute $q_*p^*S^j \mathcal{U}_2^{\vee}(-kH_2).$

Following the recipe of Proposition~\ref{45} for $p^*S^j \mathcal{U}_2^{\vee}(kH_2)$
we consider the weight $\beta=(j+k,k;0;0)$ and obtain from it
\begin{equation*}
\alpha + \rho = (j+k+3,k+2,1).
\end{equation*}
Assume all its entries are distinct.
Since $j \ge 0$, we have $j+k+3 > k+2$, so this means that
\begin{equation*}
\text{either}\quad
k + 2 > 1,
\qquad
\text{or}\quad
j + k + 3 > 1 > k + 2,
\qquad
\text{or}\quad
1 > j + k + 3.
\end{equation*}
Applying in each case the appropriate permutation and subtracting $\rho$ as explained in Proposition~\ref{45},
we obtain the above result.
%
%
\end{proof}

\begin{corollary} \label{full}
Except possibly for the first two bundles in the first block
\begin{equation}
\label{eq:bad1}
 (\cO(-5H_2+3H_3),  \mathcal{U}_2^{\vee}(-5H_2+3H_3))
 \in \mathrm{D}_3
\end{equation}
and the last four bundles in the last block
\begin{equation}
\label{eq:bad2}
(\cO(6H_3), \mathcal{U}_2^{\vee}(6H_3), S^2\mathcal{U}_2^{\vee}(6H_3), \cO(H_2+6H_3))
\in \mathrm{D}_6,
\end{equation}
the direct images of all bundles from the collection~\eqref{22} lie in $\cD$.
\end{corollary}

\begin{proof}
All bundles from the collection \eqref{22} have the form $p^*S^j \mathcal{U}_2^{\vee}(kH_2)\otimes \mathcal{O}(iH_3)$,
where $i\in[3,6]$, $j\in [0,3]$, and $k \in [-6,1]$ with various restrictions on possible triples $(i,j,k)$,
which we do not specify explicitly.
Note, however, that for $k \in [-2,1]$ we have $j \in [0,2]$.
The pushforwards of all these bundles are computed by Lemma~\ref{direct image},
so we only need to analyze its right hand side, and check when the corresponding objects lie in $\cD$.


To start with, consider the first line of Lemma~\ref{direct image}.
Here we have $k \ge 0$, hence $j \le 2$.
The corresponding objects $\Sigma^{j+k,k}\cU_3^\vee(iH_3)$ then belong to the rectangular part of~\eqref{eq:cd},
except for the cases
\begin{itemize}
\item $k = 1$, $j = 2$, $i \in [3,5]$, or
\item $i = 6$.
\end{itemize}
In the first case the corresponding bundles $\Sigma^{3,1}\cU_3^\vee(iH_3)$ belong to $\cD$ by Corollary~\ref{4}.
In the second case the inequality $k \ge 0$ gives the last four bundles in the last block of~\eqref{22},
so this is the case of~\eqref{eq:bad2}.


Next, consider the second line of Lemma~\ref{direct image}.
In this case we have $k \le -2$ and $k \ge - j - 1 \ge -4$.
If~$k \ge -3$ the corresponding objects $\Sigma^{j-1,-k-2}\mathcal{U}_3^{\vee}((i+k+1)H_3)$ belong to the rectangular part of~\eqref{eq:cd}.
On the other hand, when $k = -4$ we automatically have $j = 3$ and~$i \in [3,6]$,
and the corresponding bundles~$\Sigma^{2,2}\mathcal{U}_3^{\vee}((i+k+1)H_3)$ belong to $\cD$ by Proposition~\ref{2}.

Finally, consider the last line of Lemma~\ref{direct image}.
In this case again the corresponding objects~$\Sigma^{-k-3,j}\mathcal{U}_3^{\vee}((i+k+1)H_3)$
belong to the rectangular part of~\eqref{eq:cd}, except for the cases
\begin{itemize}
\item $k = -6$, $j \in [1,3]$, $i = 6$, or
\item $k = -5$, $j = 2$, or
\item $k = -5$, $j \in [0,1]$, $i = 3$.
\end{itemize}
In the first case the corresponding bundles are $\Sigma^{3,3}\cU^\vee(H_3)$, $\Sigma^{3,2}\cU^\vee(H_3)$, and~$\Sigma^{3,1}\cU^\vee(H_3)$,
and they belong to~$\cD$ by Proposition~\ref{10} and Corollary~\ref{4}.
In the second case the corresponding bundles $\Sigma^{2,2}\cU_3^\vee((i-4)H_3)$ belong to $\cD$ by Proposition~\ref{2} since $i\in[3,6]$.
The last case gives the first two bundles in the first block of~\eqref{22}, so this is the case of~\eqref{eq:bad1}.
%
\end{proof}

Now we are ready for the proof of the theorem.

\begin{proof}[Proof of Theorem~\textup{\ref{intro:main}}]

We prove that the category $\cD$ defined by~\eqref{eq:cd} is equal to $\Db(\IGr(3,V))$.
In view of Propositions~\ref{proposition:ec-t} and~\ref{proposition:ec} this will prove the theorem.

As we have seen above, the direct images of all bundles from the full exceptional collection~\eqref{22} in~$\mathrm{D}^{b}(\mathrm{IFl}(2,3,V))$
are contained in $\cD$ except possibly for the two collections~\eqref{eq:bad1} and~\eqref{eq:bad2}
that sit at the opposite ends of~\eqref{22}.
Applying Proposition~\ref{proposition:long-mutations}
we can move the subcollection~\eqref{eq:bad1} from the left end of~\eqref{22} to its right end,
twisting it by the anticanonical class of $\mathrm{IFl}(2,3,V)$ which is equal to~$\cO(-3H_2-4H_3)$.
The resulting collection can be written as
\begin{multline}
\label{eq:new-collection-ifl}
\Big\langle
p^*\fG_{[3,24]}\otimes \cO(3H_3),
p^*\fG\otimes \cO(4H_3),
p^*\fG\otimes \cO(5H_3),
p^*\fG'_{[1,20]} \otimes \cO(6H_3),\\
\cO(6H_3), \mathcal{U}_2^{\vee}(6H_3), S^2\mathcal{U}_2^{\vee}(6H_3), \cO(H_2+6H_3),
\cO(-2H_2+7H_3)), \mathcal{U}_2^{\vee}(-2H_2+7H_3)
\Big\rangle,
\end{multline}
where $\fG_{[3,24]}$ and $\fG'_{[1,20]}$ denote the subcollections
of the last 22 objects in~\eqref{21} and the first 20 objects in~\eqref{21prime}, respectively.


%
%
Let us denote by $\mathcal{C}$ the subcategory in  $\mathrm{D}^{b}(\mathrm{IFl}(2,3,V))$
generated by the second line in~\eqref{eq:new-collection-ifl}, i.e.,
\begin{equation*}
\mathcal{C}:= \langle \cO(6H_3), \mathcal{U}_2^{\vee}(6H_3), S^2\mathcal{U}_2^{\vee}(6H_3), \cO(H_2+6H_3),
\cO(-2H_2+7H_3)), \mathcal{U}_2^{\vee}(-2H_2+7H_3)  \rangle.
\end{equation*}
Then the first line is equal to $\mathcal{C}^\perp$, and we can rewrite~\eqref{eq:new-collection-ifl} as
\begin{equation*}
\mathrm{D}^b(\mathrm{IFl}(2,3,V))=\langle \mathcal{C}^{\perp}, \mathcal{C}  \rangle.
\end{equation*}
By Corollary \ref{full} we have
\begin{equation*}
q_*(\mathcal{C}^{\perp}) \subset \mathcal{D}.
\end{equation*}

Assume that the Lefschetz collection~\eqref{1} is not full.
Then there exists an object $G {} \in \Db(\IGr(3,V))$ left orthogonal to all bundles in the collection \eqref{1}, hence to $\cD$.
Therefore,
\begin{equation*}
\mathrm{Ext}^{\bullet}_{\mathrm{IGr}(3,V)}(G,q_*(\mathcal{C}^{\perp}))=0.
\end{equation*}
By adjunction
\begin{equation*}
\mathrm{Ext}^{\bullet}_{\mathrm{IFl}(2,3,V)}(q^*G,\mathcal{C}^{\perp}) = 0,
\end{equation*}
which means that $q^*G$ belongs to the full subcategory $\mathcal{C}$.
We conclude that
\begin{equation} \label{7}
G\simeq q_*q^*G\in q_*\mathcal{C}.
\end{equation}
By definition of $\mathcal{C}$ together with Lemma~\ref{direct image}, we have
\begin{equation*}
q_*\mathcal{C} =
\langle \cO(6H_3), \mathcal{U}_3^{\vee}(6H_3), S^2\mathcal{U}_3^{\vee}(6H_3), \Lambda^2\cU^{\vee}(6H_3)\rangle.
\end{equation*}
Note that this is contained in $\mathfrak{E}(6)$.
In particular, we conclude that
\begin{equation*}
G \in \mathfrak{E}(6).
\end{equation*}
At the same time, we have $G \in {}^\perp\mathfrak{E}$ by definition.
It remains to note that
\begin{equation} \label{6}
^{\perp}\mathfrak{E}\cap \mathfrak{E}(6)=0.
\end{equation}
Indeed, every nonzero object in $\mathfrak{E}(6)$ is of the form $E(6),$ where $0\ne E\subset \mathfrak{E}$,
and using Serre duality
\begin{equation*}
\mathrm{Hom}^{\bullet}(E(6),E)= \mathrm{Hom}^{\bullet}(E,E)^\vee \ne 0
\end{equation*}
we conclude that $E(6)$ can not lie in $^{\perp}\mathfrak{E}.$

Altogether, this proves that $G = 0$, hence the exceptional collection~\eqref{1} is full.
\end{proof}

%

\section{Applications}
\label{section:applications}

In this section we collect several simple consequences of Theorem~\ref{intro:main}.

\subsection{Exceptional collections on flag varieties of $\Sp(8)$}

The first application is straightforward.

\begin{theorem}
For any parabolic subgroup $P \subset \Sp(8)$ the flag variety $\Sp(8)$ has
a full exceptional collection of $\Sp(8)$-equivariant vector bundles.
\end{theorem}
\begin{proof}
By~\cite[Section~1.2]{1} it is enough to prove the theorem for any maximal parabolic subgroup.
In Theorem~\ref{intro:main} we established the case~$(C_4,P_3)$.
For the cases corresponding to $P_1$ and $P_2$, i.e., the cases of~$\mathbb{P}^7$ and $\mathrm{IGr}(2,8)$, see~\cite{6,2},
and for the case corresponding to $P_4$, i.e., $\mathrm{IGr}(4,8)$, see~\cite{11}.
\end{proof}

\subsection{Fractional Calabi–Yau categories}

By Theorem~\ref{intro:main} the exceptional collection
\begin{equation*}
(F,\mathfrak{E},\mathfrak{E}(1),\mathfrak{E}(2),F(3),\mathfrak{E}(3),\mathfrak{E}(4),\mathfrak{E}(5))
\end{equation*}
is a rectangular Lefschetz collection with respect to $\mathcal{O}(3)$.
Using the results of~\cite{10} we construct two new examples of fractional Calabi--Yau categories.

\begin{definition} [\cite{10}]
A triangulated category $\mathcal{T}$ is a fractional Calabi–Yau category if
it has a Serre functor $\mathbf{S}_{\mathcal{T}}$ and there are integers $p$ and $q\ne 0$ such that $\mathbf{S}_{\mathcal{T}}^q\simeq [p].$
\end{definition}

Denote by~$\mathcal{B}$
the first block~$\mathcal{B}:= (F,\mathfrak{E},\mathfrak{E}(1),\mathfrak{E}(2))$
of the above Lefschetz collection.

\begin{proposition}[{\cite[Theorem $3.5$ and Corollary $3.7$]{10}}]
Suppose that a map $f\colon Y \to \mathrm{IGr}(3,V)$ is a divisorial embedding
with the image $f(Y)$ being a divisor in the linear system $\cO(3)$.
Then the functor $f^*\colon \mathrm{D}^b(\mathrm{IGr}(3,V))\to \mathrm{D}^b (Y)$ is fully faithful on $\mathcal{B}$ and
induces a semiorthogonal decomposition
\begin{equation*}
\mathrm{D}^b (Y)=\langle \mathcal{A}_Y, \mathcal{B}_Y  \rangle,
\end{equation*}
where $\mathcal{B}_Y= f^*\mathcal{B}$ and $\mathcal{A}_Y$ is the orthogonal subcategory.
Moreover, $\mathcal{A}_Y$ is a fractional Calabi–Yau category with the Serre functor $\mathbf{S}_{\mathcal{A}_Y}=[9].$
\end{proposition}

\begin{proposition}[{\cite[Theorem $3.5$ and Corollary $3.8$]{10}}]
Suppose that a map $f\colon Y \to \mathrm{IGr}(3,V)$ is a double covering branched in a divisor in the linear system  $\cO(6)$.
Then the functor $f^*\colon \mathrm{D}^b(\mathrm{IGr}(3,V))\to \mathrm{D}^b (Y)$ is fully faithful on $\mathcal{B}$ and
induces a semiorthogonal decomposition $$\mathrm{D}^b (Y)=\langle \mathcal{A}_Y, \mathcal{B}_Y  \rangle, $$
where $\mathcal{B}_Y= f^*\mathcal{B}$ and $\mathcal{A}_Y$ is the orthogonal subcategory.
Moreover, $\mathcal{A}_Y$ is a fractional Calabi–Yau category with the Serre functor
$\mathbf{S}_{\mathcal{A}_Y}=\tau[11] $, where $\tau$ is the involution of the covering.
\end{proposition}

\subsection{Residual category}

In this subsection we compute the residual category for the first two exceptional collections of Theorem~\ref{intro:main}.
Let us first recall the definition.

\begin{definition}[\cite{KS}]
Let $E_\bullet$ be a Lefschetz exceptional collection in $\Db(X)$ with respect to a line bundle~$\cO_X(1)$
with the support partition~$\lambda$ (see Definition~\ref{def:lefschetz}).
The subcategory of $\mathrm{D}^{b}(X)$ orthogonal to its rectangular part is called its \textbf{residual category}:
$$\mathcal{R}=\langle E_{1}, E_{2},\ldots,E_{\lambda_{i-1}},E_{1}(1), E_{2}(1),\ldots,E_{\lambda_{i-1}}(1),\ldots,E_{1}(i-1), E_{2}(i-1),\ldots,E_{\lambda_{i-1}}(i-1) \rangle^{\perp}.$$
\end{definition}

The next theorem supports Conjecture 1.11 from \cite{KS}.

\begin{theorem} \label{residual}
The residual category $\cR$ for the Lefschetz decomposition
\begin{equation}
\label{47}
\Db(\IGr(3,V)) = \langle F,\mathfrak{E}, F(1),\mathfrak{E}(1),\mathfrak{E}(2),\mathfrak{E}(3),\mathfrak{E}(4),\mathfrak{E}(5)\rangle
\end{equation}
is generated by two completely orthogonal exceptional objects.
\end{theorem}
\begin{proof}
By definition, $\cR$ is the orthogonal to the collection $\fE, \fE(1),\fE(2),\fE(3),\fE(4),\fE(5)$ in $\Db(\IGr(3,V))$.
Therefore, it is generated by the exceptional pair $(F, \mathbb{L}_{\mathfrak{E}}F(1))$, i.e.,
\begin{equation*}
\mathcal{R} \simeq \langle F,\mathbb{L}_{\mathfrak{E}}F(1)\rangle.
\end{equation*}
So, it remains to prove that this pair is completely orthogonal.
One semiorthogonality is evident.
Furthermore, by Corollary~\ref{FF} we have $\mathbb{L}_{\mathfrak{E}}F(1)\cong\mathbb{R}_{\mathfrak{E}(-1)}F(-1)$ up to a shift.
Since~\eqref{47} is an exceptional collection we have $\mathrm{Ext}^{\bullet}(F,\mathfrak{E}(-1))=0$ and $\mathrm{Ext}^{\bullet}(F,F(-1))=0,$ so we conclude that~$\mathrm{Ext}^{\bullet}(F,\mathbb{L}_{\mathfrak{E}}F(1))=\mathrm{Ext}^{\bullet}(F,\mathbb{R}_{\mathfrak{E}(-1)}F(-1))=0$ and we get the statement.
\end{proof}


\end{document}